\documentclass[12pt]{scrartcl}

\newcommand{\mytitle}{A study of large fringe and non-fringe subtrees in conditional Galton-Watson trees}
\newcommand{\myauthors}{Xing Shi Cai, Luc Devroye}

\usepackage[utf8]{inputenc}
\usepackage[numbers]{natbib}
\usepackage{amsmath}
\usepackage{amsthm}
\usepackage{amssymb}
\usepackage{bm}
\usepackage{myprob}
\usepackage{mythm}
\usepackage{mymath}
\usepackage{bbold}
\usepackage{array}
\usepackage{float}
\usepackage[pdftitle={\mytitle},
    pdfauthor={\myauthors},
    hidelinks]{hyperref}
\usepackage{pstricks}
\usepackage{enumitem}
\usepackage{manfnt}
\usepackage{soul}
\usepackage{tabu}
\usepackage{colortbl}

\usepackage{tikz}
\usetikzlibrary{calc}

\newcommand{\drafVersion}{Full Version}
\usepackage{mathabx} 

\hypersetup{bookmarksnumbered}

\newcommand\widenonf[1]{\widebar{#1}}

\newcommand{\frT}{{\mathfrak{T}}}
\newcommand{\tset}{\frT}

\newcommand{\gwraw}{\cT}
\newcommand{\gwsup}{\mathbf{gw}}
\newcommand{\gw}{\gwraw^{\gwsup}}
\newcommand{\gwrand}{{\gw_{n,*}}}
\newcommand{\gwn}{\gw_n}

\newcommand{\tildegw}{\widetilde{\gwraw}^{\gwsup}}

\newcommand{\overgwn}{\widenonf{\gwraw}_n^{\gwsup}}
\newcommand{\tildegwn}{\widetilde{\gwraw}_n^{\gwsup}}

\newcommand{\gwbi}{\gwraw^{\mathbf{bin}}}
\newcommand{\gwbin}{\gwbi_{n}}
\newcommand{\Tnbin}{T_{n}^{\mathbf{bin}}}

\newcommand{\tsetkn}{\cS_{k_n}}
\newcommand{\tsetk}{\cS_{k}}
\newcommand{\tsetplus}[1]{\cS_{\le #1}^{+}}
\newcommand{\tsetpluskn}{\tsetplus{k_n}}

\newcommand{\treecount}{N}
\newcommand{\fringeT}{\treecount_{T}(\gwn)}
\newcommand{\fringeTn}{\treecount_{T_n}(\gwn)}
\newcommand{\fringeS}{\treecount_{\cS}(\gwn)}
\newcommand{\fringeSk}{\treecount_{\tsetk}(\gwn)}
\newcommand{\fringeSkn}{\treecount_{\tsetkn}(\gwn)}

\newcommand{\bfringeT}{\treecount_{T}(\overgwn)}
\newcommand{\bfringeTn}{\treecount_{T_n}(\overgwn)}

\newcommand{\nfsup}{{\mathbf{nf}}}
\newcommand{\nftreecount}{\treecount^\nfsup}
\newcommand{\nftcount}{\nftreecount_{T}(\gwn)}
\newcommand{\nftncount}{\nftreecount_{T_n}(\gwn)}
\newcommand{\nfpi}{\pi^\nfsup}

\newcommand{\tleast}{{T^{{\mathbf{min}}}_k}}
\newcommand{\pleast}{p^{{\mathbf{min}}}}
\newcommand{\pleastk}{{\pleast_k}}
\newcommand{\pleastToOneByk}{{(\pleastk)^{1/k}}}
\newcommand{\pmax}{p_{{\mathbf{max}}}}
\newcommand{\ileast}{\kappa}

\newcommand{\treerary}{T^{\mathbf{r}\text{-}\mathbf{ary}}}
\newcommand{\tstar}{T^{\mathbf{star}}}

\newcommand{\pSetMaxSize}{K_n}

\DeclareMathOperator{\rootat}{\prec}

\newcommand{\xin}{\xi^{\mathbf{n}}}
\newcommand{\xinvec}{{\bm \xi}^{\mathbf{n}}}
\newcommand{\xinlist}{\xin_{1}, \xin_{2}, \ldots, \xin_{n}}
\newcommand{\xilist}{\xi_{1}, \xi_{2},\ldots, \xi_{n}}

\newcommand{\xitilden}{\widetilde{\xi}^{\mathbf{n}}}
\newcommand{\xitildenvec}{\widetilde{\bm \xi}^{\mathbf{n}}}
\newcommand{\xitildenlist}{\xitilden_{1}, \xitilden_{2}, \ldots, \xitilden_{n}}

\newcommand{\xihatn}{\widehat{\xi}^{\mathbf{n}}}
\newcommand{\xihatnvec}{\widehat{\bm \xi}^{\mathbf{n}}}
\newcommand{\xihatnlist}{\xihatn_{1}, \xihatn_{2}, \ldots, \xihatn_{n}}

\newcommand{\xiovern}{\widenonf{\xi}^{\mathbf{n}}}
\newcommand{\xiovernvec}{\widenonf{\bm \xi}^{\mathbf{n}}}
\newcommand{\xiovernlist}{\xiovern_{1}, \xiovern_{2}, \ldots, \xiovern_{n}}

\newcommand{\Itilde}{\widetilde{I}}
\newcommand{\Jtilde}{\widetilde{J}}

\newcommand{\Ihat}{\widehat{I}}
\newcommand{\Jhat}{\widehat{J}}

\newcommand{\Iover}{\widenonf{I}}

\newcommand{\Gammaover}{\widenonf{\Gamma}}

\newtheorem{myCond}{Condition}

\newcommand{\Toplus}{\{T\boxplus T\}}
\newcommand{\ToplusB}{\cA_{\mathbf{big}}}
\newcommand{\ToplusS}{\cA_{\mathbf{small}}}

\newcommand{\idx}{\cI}
\newcommand{\idxVecd}[1]{\idx_{#1}(\vecd)}
\newcommand{\idxVecdIn}[1]{\idx^{\mathbf{in}}_{#1}(\vecd)}
\newcommand{\idxVecdOut}[1]{\idx^{\mathbf{out}}_{#1}(\vecd)}

\newcommand{\splay}{\mathop{\mathrm{Splay}}}

\title{\mytitle}
\subtitle{\drafVersion{}}
\author{\small \myauthors
    \\
    \small
    School of Computer Science, McGill University of Montreal, Canada,\\
    \small
    \texttt{xingshi.cai@mail.mcgill.ca}\\
    \small
    \texttt{lucdevroye@gmail.com}
}

\date{\small \today}

\begin{document}

\maketitle

\begin{abstract}
We study the conditions for families of fringe or non-fringe
subtrees to exist with high probability (whp) in \(\gwn\), a Galton-Walton tree of size
\(n\). We first give a Poisson approximation of fringe subtree counts in \(\gwn\),
which permits us to determine the height of the maximal complete $r$-ary fringe subtree.
Then we determine the maximal \(K_n\) such that
every tree of size at most \(K_n\)
appears as fringe subtree in \(\gwn\) whp.
Finally, we show that non-fringe subtree counts are concentrated and determine, as an application, the
height of the maximal complete \(r\)-ary non-fringe subtree in \(\gwn\).
\end{abstract}


\section{Introduction}

In this paper, we study the conditions for families of fringe or non-fringe subtrees to exist whp
(with high probability) in a Galton-Walton tree conditional to be of size \(n\).
In particular, we want to find the height of the maximal complete \(r\)-ary fringe and non-fringe subtrees.
We also want to determine the threshold \(k_n\) such that all trees of size at most
\(k_n\) appear as fringe subtrees. In doing so, we extend \citeauthor{Janson2014Asym}'s
\citep{Janson2014Asym} result on fringe subtrees counts and prove a new concentration theorem for
non-fringe subtree counts.

Let \(\tset\) be the set of all rooted, ordered, and unlabeled trees, which we refer to as
\emph{plane} trees.  All trees considered
in this paper belong to \(\tset\).  (See \citet[sec.\ 2.1]{janson2012simply} for details.)
Figure \ref{intro:fig:plane:tree} lists all plane trees of at most \(4\) nodes.
\begin{figure}[ht!]
  \centering
    \begin{tikzpicture}
    \node[anchor=south west,inner sep=0] at (0,0) {\includegraphics{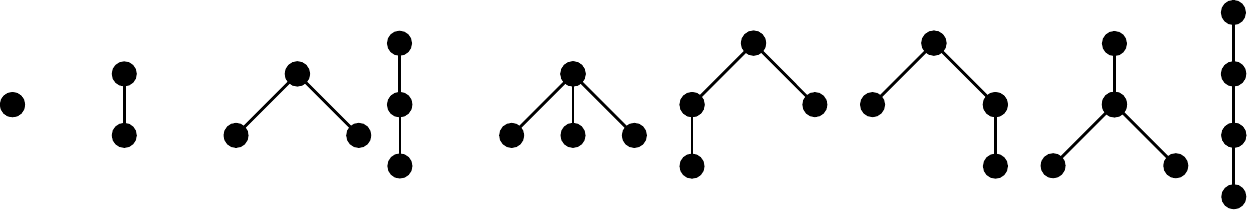}};
    \end{tikzpicture}
    \caption{Plane trees of at most \(4\) nodes.}
    \label{intro:fig:plane:tree}
\end{figure}

Given a tree \(T \in \tset\) and a node \(v \in T\), let \(T_v\) denote the subtree rooted
at \(v\).  We call \(T_v\) a \emph{fringe} subtree of
\(T\).  If \(T_v\) is isomorphic to some tree \(T' \in \tset\), then we write \(T' = T_v\) and say that \(T\) has a
\emph{fringe} subtree of shape \(T'\) rooted at \(v\), or simply \(T\) contains \(T'\) as a fringe
subtree.

On the other hand, if \(T_v\) can be made isomorphic to
\(T'\) by replacing some or none of its own fringe subtrees with leaves (nodes without children), then we write \(T' \rootat
T_v\) and say that \(T\) has a \emph{non-fringe} subtree of shape \(T'\) rooted at \(v\), or
simply \(T\) contains \(T'\) as a non-fringe subtree.
(Note that \(T' = T_v\) implies that \(T' \prec T_v\).)
We also use the notation \(T' \rootat T\) to denote that \(T\) has a non-fringe subtree of shape
\(T'\) at its root.
Figure \ref{intro:fig:subtrees} shows some examples of fringe and non-fringe subtrees.

\begin{figure}[ht!]
  \centering
    \begin{tikzpicture}
    \node[anchor=south west,inner sep=0] at (0,0) {\includegraphics{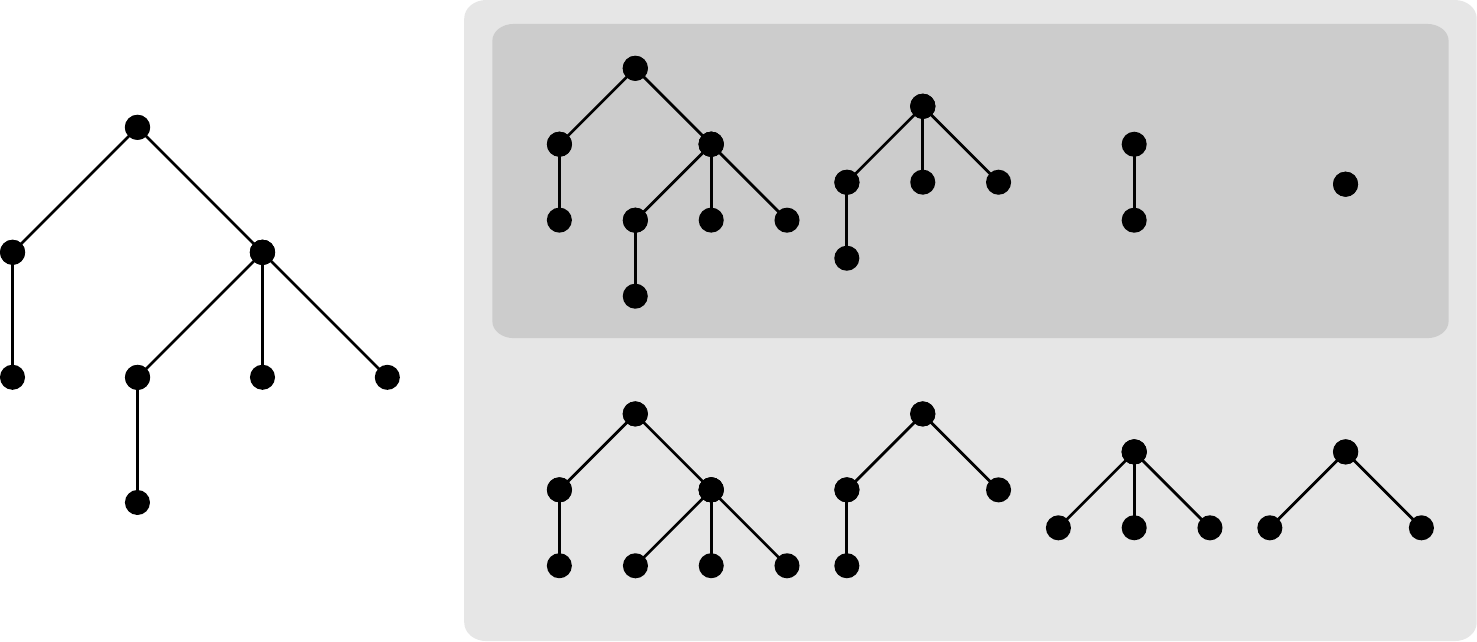}};
    \node at (2,0.7) {$T$};
    \node at (12,3.5) {fringe subtrees of $T$};
    \node at (12,0.5) {non-fringe subtrees of $T$};
    \end{tikzpicture}
    \caption{Examples of fringe and non-fringe subtrees.}
    \label{intro:fig:subtrees}
\end{figure}
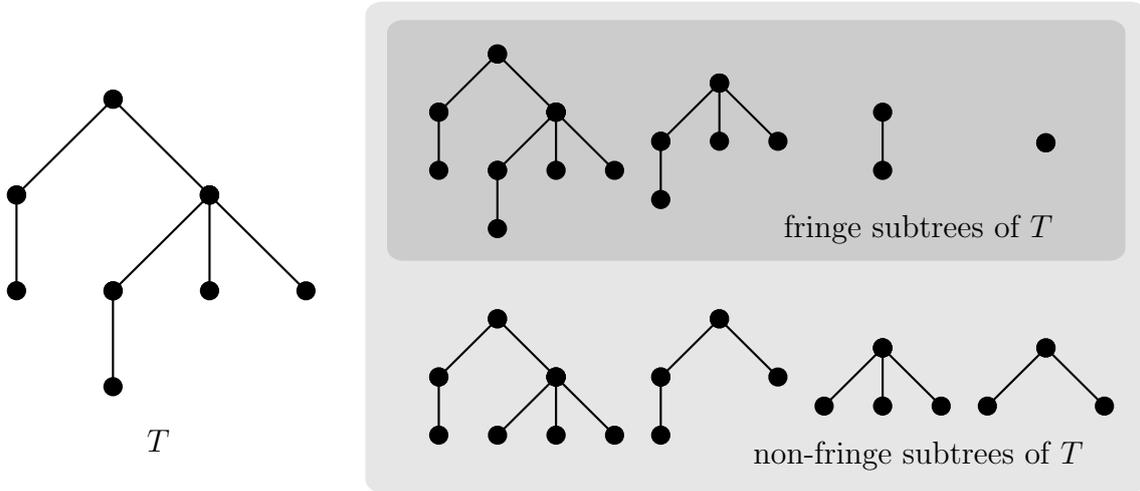

Let \(\xi\) be a non-negative integer-valued random variable.  The Galton-Watson tree \(\gw\)
with
offspring distribution \(\xi\) is the random tree generated by starting from the root and
independently giving each node a random number of children, where the numbers of children are all
distributed as \(\xi\).  The conditional Galton-Watson tree \(\gwn\) is \(\gw\) restricted to the
event \(|\gw|=n\), i.e., \(\gw\) has \(n\) nodes.  The comprehensive survey by
\citet{janson2012simply} describes the history and the basic properties of these trees.

In the study of conditional Galton-Watson trees, the following is usually assumed throughout the
paper:
\begin{myCond}
    \label{cond:offspring}
    Let \(\gwn\) be a conditional Galton-Watson tree of size \(n\) with offspring distribution
    \(\xi\), such that \(\e \xi = 1\) and \(0 < \sigma^{2} := \V{\xi} < \infty\).
    Let \(\gw\) be the corresponding unconditional Galton-Watson tree.
\end{myCond}

We summarize notations:
\begin{itemize}[label={$\cdot$},leftmargin=*]
\item  \(\tset\)  --- the set of all rooted, ordered and unlabeled trees (plane trees)
\item  \(T\)      --- a tree in \(\tset\)
\item  \(T_v\)   --- a fringe subtree of \(T\) rooted at node \(v \in T\)
\item  \(\xi\)   --- a non-negative integer-valued random variable with \(\e \xi = 1\) and \(0 < \sigma^2 := \V{\xi} < \infty\)
\item  \(h\)     --- the span of \(\xi\), i.e., \({\gcd}\{i \ge 1: p_i > 0\}\)
\item  \(\gw\)   --- an unconditional Galton-Watson tree  with offspring distribution \(\xi\)
\item  \(\gwn\)   --- \(\gw\) given that \(|\gw|=n\)
\item  \(\gw_{n,v}\)   --- a fringe subtree of \(\gwn\) rooted at node \(v \in \gwn\)
\item  \(\gw_{n,*}\)   --- a fringe subtree of \(\gwn\) rooted at a uniform random node of \(\gwn\)
\item  \(T_n\)   --- a sequence of trees
\item  \(\cS_{n}\)   --- the set of all trees of size \(n\)
\item  \(\cS\)   --- a set of trees
\item  \(\tsetplus{n}\)   --- the set \(\{T \in \tset: |T| \le n, \p{\gw = T} > 0\}\)
\item  \(\cA_n\)   --- a sequence of sets of trees
\item  \(\fringeS\) --- the number of fringe subtrees of \(\gwn\) that belong to \(\cS\)
\item  \(\pi(\cS)\) --- \(\p{\gw \in \cS}\)
\item  \(\nftcount\) --- the number of non-fringe subtrees of \(\gwn\) of shape \(T\)
\item  \(\nfpi(T)\) --- \(\p{T \rootat \gw}\), the probability that \(\gw\) has a non-fringe subtree \(T\) at its root
\end{itemize}

If \(p_1 = 0\), then there exist positive integers \(n\) such that \(\p{|\gw|=n} = 0\).
For such \(n\), \(\gwn\) is not well-defined.
But it is easy to show that \(\p{|\gw|=n} > 0\) for all \(n \ge n_0\) with \(n - 1 \equiv 0 \pmod h\),
where \(h\) is span of \(\xi\) and \(n_0\) depends only on \(\xi\) (see \citep[cor.\ 15.6]{janson2012simply}).
Therefore, in this paper, for all asymptotic results about \(\gw\) and \(\gwn\), the limits are always taken
along the subsequence with \(n - 1\equiv 0 \pmod h\).

Extending a result by \citet{aldous1991asymptotic},
\citeauthor{janson2012simply} \cite[thm.\ 7.12]{janson2012simply} proved the following theorem:
\begin{theorem}
    \label{intro:thm:count}
    Assume Condition \ref{cond:offspring}.
    The conditional distribution \(\cL(\gwrand|\gwn)\) converges in probability to
    \(\cL(\gw)\). In other words, for all \(T \in \tset\), as \(n \to \infty\),
    \begin{equation}
        \frac{\treecount_{T}(\gwn)}{n} = \p{\gwrand = T | \gwn} \inprob \p{\gw = T}.
    \end{equation}
\end{theorem}

Later \citeauthor{Janson2014Asym} strengthened the above result,
proving the asymptotic normality of \(\treecount_T(\gwn)\) by studying additive functionals on \(\gwn\)
\cite{Janson2014Asym}.

A natural generalization of \(\treecount_{T}(\gwn)\) is to consider fringe subtree counts
\(\fringeTn\) where \(T_n \in \tset\) is a sequence of trees instead of a fixed tree \(T\).
Let \(\Po(\lambda)\) denote
a Poisson random variable with mean \(\lambda\).
We have:
\begin{theorem}
    \label{intro:thm:fringe:count}
    Assume Condition \ref{cond:offspring}.
    Let \(\pi(T) := \p{\gw = T}\) and let \(k_n \to \infty, k_n = o(n)\).
    Then
    \begin{equation}
        \lim_{n \to \infty} \sup_{T:|T|=k_n}
        \distTV{\fringeT, \Po(n \pi(T))} = 0,
        \label{eq:tree:poi}
    \end{equation}
    where \(\distTV{\,\cdot\,, \,\cdot\,}\) denotes the total variation distance.
    Therefore,
    letting \(T_n\) be a sequence of trees with \(|T_n|=k_n\), we have as \(n \to \infty\):
    \begin{enumerate}[label=(\roman*)]
        \item
            If \(n \pi(T_n) \to 0\), then \(\fringeTn = 0\) whp.
        \item If \(n \pi(T_n) \to \mu \in (0, \infty)\), then
                \(\fringeTn \inlaw \Po(\mu)\).
        \item If \(n \pi(T_n) \to \infty\), then
            \[
                \frac{\fringeTn - n \pi(T_n)}{\sqrt{n \pi(T_n)}} \inlaw N(0,1),
            \]
            where \(N(0,1)\) denotes the standard normal distribution, and \(\inlaw\) denotes
            convergence in distribution.
    \end{enumerate}
\end{theorem}

In Theorem \ref{intro:thm:fringe:count}, (i)--(iii) follow directly from \eqref{eq:tree:poi} by the following lemma,
whose simple proof we omit:
\begin{lemma}
    \label{intro:lem:poisson:distribution}
    Let \(X_n\) be a sequence of random variables.
    Let \(\mu_n\) be a sequence of non-negative real numbers.
    Assume that \(\distTV{X_n, \Po(\mu_n)} \to 0\) as \(n \to \infty\).
    We have:
    \begin{enumerate}[label=(\roman*)]
        \item
                If \(\mu_n \to 0\), then \(X_n = 0\) whp.
        \item If \(\mu_n \to \mu \in (0, \infty)\), then
                \(X_n \inlaw \Po(\mu)\).
        \item If \(\mu_n \to \infty\), then
            \(
            ({X_n - \mu_n})/{\sqrt{\mu_n}} \inlaw N(0,1).
            \)
    \end{enumerate}
\end{lemma}

Theorem \ref{intro:thm:fringe:count} can be generalized as follows:
\begin{theorem}
    \label{intro:thm:fringe:set:count}
    Assume Condition \ref{cond:offspring}.
    Let \(\tsetkn\) be the set of all trees of size \(k_n\),
    where \(k_n \to \infty\) and \(k_n = o(n)\).
    For \(\cS \subseteq \tsetkn\), let \(\pi(\cS):= \p{\gw \in \cS}\)
    and \(\treecount_{\cS}(\gwn) := \sum_{v \in \gwn} \iverson{\gw_{n,v} \in \cS}\).
    Then
    \[
        \lim_{n \to \infty} \sup_{\cS \subseteq \tsetkn} \distTV{\treecount_\cS(\gwn), \Po(n
        \pi(\cS))} = 0.
    \]
    Therefore, letting \(\cA_n\) be a sequence of sets of trees with \(\cA_n \subseteq \tsetkn\),
    we have:
    \begin{enumerate}[label=(\roman*)]
        \item
            If \(n \pi(\cA_n) \to 0\), then \(\treecount_{\cA_n}(\gwn) = 0\) whp.
        \item If \(n \pi(\cA_n) \to \mu \in (0, \infty)\), then
            \(\treecount_{\cA_n}(\gwn) \inlaw \Po(\mu)\).
        \item
            If \(n \pi(\cA_n) \to \infty\), then
            \[
                \frac{\treecount_{\cA_n}(\gwn) - n \pi(\cA_n)}{\sqrt{n \pi(\cA_n)}} \inlaw
                N(0,1).
            \]
    \end{enumerate}
\end{theorem}

The proof of Theorem \ref{intro:thm:fringe:count} is given in Section \ref{sec:fringe}.
It uses many ingredients from previous results on fringe subtrees, especially from
\citet{Janson2014Asym}. (In particular, Lemma 6.2 of \citep{Janson2014Asym} makes the computation of
the variance of \(N_{\cS_n}(\gwn)\) quite easy, which is crucial for the proof.)
The key step of the proof is based on {subtree switching}. It permits the coupling of
two conditional Galton-Watson trees. Then we can apply the exchangeable pair method by
\citet{chatterjee2005}, which is an extension of Stein's method \cite{Stein1972}, to study the
convergence of total variation distance.
This approach can easily be modified to prove Theorem \ref{intro:thm:fringe:set:count}, whose steps we sketch
at the end of Section \ref{sec:fringe}.

\emph{Binary search trees} and \emph{recursive trees} are also well-studied random tree models (see
\citet{drmota2009random}).  Many authors have found results similar to
Theorem \ref{intro:thm:fringe:count} for these two types of trees, see, e.g., \cite{Feng2008, Fuchs2008, devroye1991limit, devroye2002limit, Flajolet1997}.
For recent developments, see \citet{holmgren2014limit}.

We say that a tree \(T\) is \emph{possible} if \(\p{\gw = T} > 0\).  As an application of Theorems
\ref{intro:thm:fringe:count} and \ref{intro:thm:fringe:set:count}, we ask the following
question --- when does \(\gwn\) contain all possible trees within a family of trees (possibly
depending on \(n\)). As shown in subsection \ref{sec:fringe:app:coupon}, this is essentially a variation
of the coupon collector problem.

In subsection \ref{sec:fringe:app:r:ary} we answer the above question for the set of complete \(r\)-ary
trees.  Let \(H_{n,r}\) be the maximal integer such that \(\gwn\) contains all complete \(r\)-ary
trees of height at most \(H_{n,r}\) as fringe subtrees.  Lemma \ref{family:lem:r:ary} shows that
\(H_{n,r} - \log_r \log n\) converges in probability to an explicit constant.

Let \(\tsetplus{k}\) be the set of all possible trees of size at most \(k\).
Let \(\pSetMaxSize=\max\{k: \tsetplus{k} \subseteq \cup_{v \in \gwn} \gw_{n,v}\}\), i.e., \(K_n\) is
the maximal \(k\) such that every tree in \(\tsetplus{k}\) appears in \(\gwn\) as fringe subtrees.
In subsection \ref{sec:fringe:app:all:possible}, we show that, roughly speaking, if the tail of the offspring
distribution does not drop off too quickly, \(K_{n}/\log n\) converges in probability to a positive constant.
Otherwise, we have \(K_{n}/\log n \inprob 0\).
For example, for a random Cayley tree, we have \(K_n \log \log (n)/\log(n) \inprob 1\).
For many well-known Galton-Watson trees, we also give the second order asymptotic term of
\(K_n\).

Non-fringe subtrees are more complicated to analyze.
However since on average fringe subtrees in \(\gwn\) behave like unconditional Galton-Watson trees
when \(n\) is large, the number of non-fringe subtrees of shape \(T\) should be more or
less \(n \p{T \rootat \gw}\).
The following theorem is a precise version of this intuition.
\begin{theorem}
    \label{intro:thm:nonfringe:count}
    Assume Condition \ref{cond:offspring}.
    Let \(\nfpi(T) := \p{T \rootat \gw} \).
    Let \(\nftcount := \sum_{v \in \gwn} \iverson{T \rootat \gw_{n,v}}\).
    Let \(T_n\) be a sequence of trees with \(|T_n|=k_n\) where \(k_n \to \infty\) and \(k_n = o(n)\).
    We have
    \begin{enumerate}[label=(\roman*)]
        \item If \(n \nfpi(T_n) \to 0\), then \(\nftncount \inprob 0\).
        \item If \(n \nfpi(T_n) \to \infty\), then \(\nftncount/(n \nfpi(T_n)) \inprob 1\).
    \end{enumerate}
\end{theorem}

\citet{chyzak2008distribution} studied non-fringe subtrees for various random trees, including
\emph{simply generated trees}.
They proved that if for all \(n\)
we have \(T_n = T\) where \(T\) is fixed, then \(\nftncount\) has a central limit theorem.
However, Theorem \ref{intro:thm:nonfringe:count} cannot be simply derived from their result as our
\(T_n\) depends upon \(n\).

\begin{myRemark*}
    It is tempting to try to prove that if \(n \nfpi(T_n) \to \mu \in (0,\infty)\), then we have
    \(\nftreecount_{T_n}(\gwn) \inlaw \Po(\mu)\), which is true for fringe subtrees.
    Unfortunately, this is not true in general. See Lemma \ref{nonfringe:lem:line} in Section
\ref{sec:nfringe:asympt}.
\end{myRemark*}

In Section \ref{sec:nfringe}, we give a proof of Theorem \ref{intro:thm:nonfringe:count} and apply it to
study the maximal complete \(r\)-ary non-fringe subtree in \(\gwn\). The paper ends with some open questions
in Section~\ref{sec:open}.

\section{Notations and Preliminaries}

\label{sec:pre}

In this section we list some well-known results regarding conditional Galton-Watson trees and
Poisson approximations.

\subsection{Conditional Galton-Watson trees}

Given a tree \(T\) of size \(k\), let \(v_1, \ldots, v_k\) be the nodes of \(T\) in {\scshape dfs}
(depth-first-search) order.
Let \(d_i\) be the \emph{degree} (the number of children) of \(v_i\).
We call \( (d_1, d_2, \ldots, d_k)\) the \emph{preorder degree sequence} of \(T\).
Let \(\N := \{1,2,\ldots\}\) and let \(\N_{0} := \{0\} \cup \N\).
It is well-known that (see \citet[lem.\ 15.2]{janson2012simply}):
\begin{lemma}
    \label{pre:lem:tree:deg:seq}
    A sequence \( (d_1,d_2,\ldots,d_k) \in \N_{0}^k\) is the preorder degree sequence of some tree if and only if it
satisfies
\begin{equation}
    \begin{cases}
        \sum_{i=1}^{j} d_i \ge j,  \qquad \qquad (1 \le j \le k-1) \\
        \sum_{i=1}^{k} d_i = k-1.
    \end{cases}
    \label{pre:eq:deg:seq}
\end{equation}
\end{lemma}

Figure \ref{pre:fig:tree:deg:seq} gives a demonstration of Lemma \ref{pre:lem:tree:deg:seq}.
\begin{figure}[ht!]
  \centering
    \begin{tikzpicture}
    \node[anchor=south west,inner sep=0] at (0,0) {\includegraphics{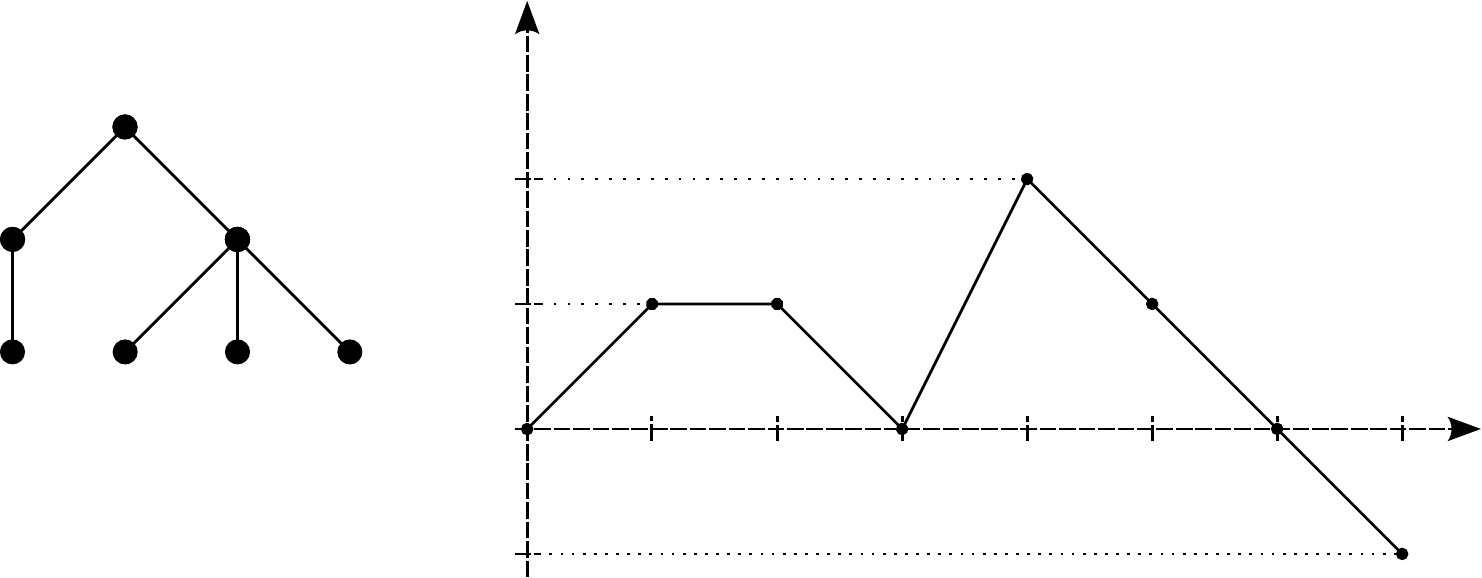}};
    \node at (2,1.2) {\small The degree sequence};
    \node at (2,0.7) {\small $(d_1, \ldots,d_7)=(2,1,0,3,0,0,0)$};
    \node at (12,5) {\small $\sum_{i=1}^j d_i-j$ for $j=1,\ldots,7$};
    \coordinate (x) at (5.4,0.8);
    \coordinate (y) at (4.9,1.5);
    \foreach \x in {0,...,6}
       {
           \pgfmathtruncatemacro{\label}{\x+1};
           \pgfmathsetmacro{\xshift}{1.26*(\x+1)};
           \node at ($(x)+(\xshift,0)$) {\small \label};
        }
    \foreach \y in {-1,...,3}
       {
           \pgfmathtruncatemacro{\label}{\y};
           \pgfmathsetmacro{\yshift}{1.26*(\y)};
           \node at ($(y)+(0,\yshift)$) {\small \label};
        }
    \end{tikzpicture}
    \caption{Example of preorder tree degree sequence.}
    \label{pre:fig:tree:deg:seq}
\end{figure}
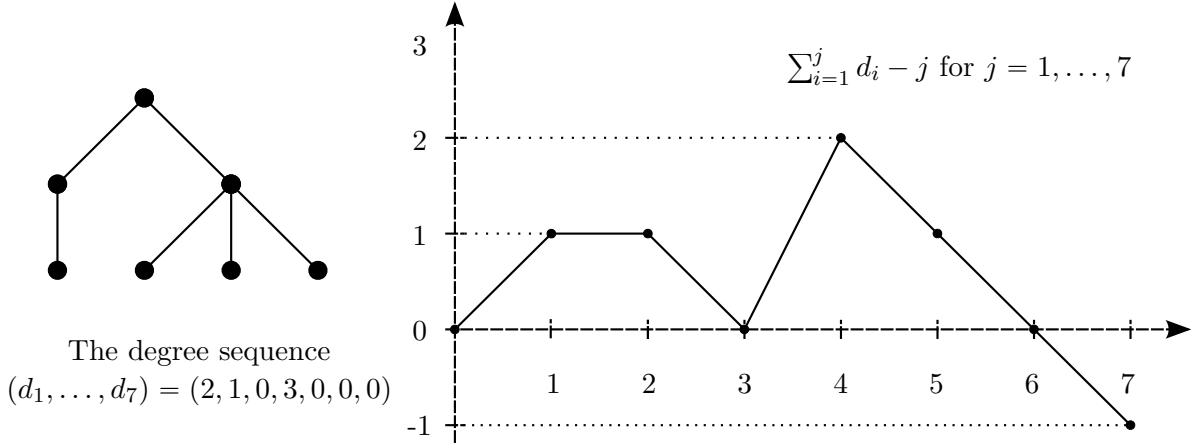

Let \(\cD_{k} \subseteq \N_0^k\) be the set of all preorder degree sequences of length \(k\).
Observe:
\begin{corollary}
    \label{cor:pre:tree:deg:seq}
If \( (d_1, d_2, \ldots, d_k) \in \cD_k \), then it is impossible that
there exists \(1 \le k' < k\) such that \( (d_1, d_2, \ldots, d_{k'}) \in \cD_{k'}\).
\end{corollary}

Let \( \xinvec := (\xin_1, \xin_2,\ldots, \xin_n) \) be the preorder degree sequence of \(\gwn\).
Let \( \xitildenvec := (\xitilden_1, \xitilden_2,\ldots, \xitilden_n)\) be a uniform random cyclic
rotation of
\( \xinvec \).
Let \(\xi_{1}, \xi_2,\ldots\) be i.i.d.\ copies of \(\xi\).
Let \(S_n := \sum_{i=1}^n \xi_{i}\).
The next lemma is a well-known connection between \( \xitildenvec\) and \( \xilist\)
(see, e.g., \citet{otter1948number}, \citet{Kolchin1986random}, \citet{dwass1969total}
and \citet{pitman1998enumerations}).
For a
complete proof, see \citet[cor.\ 15.4]{janson2012simply}.

\begin{lemma}
    \label{pre:lem:tree:deg:seq:rotation}
    Assume that \(\p{|\gw|=n} > 0\).
    We have
    \[
        (\xitilden_1, \xitilden_2,\ldots, \xitilden_n ) \eql \left(\xilist~|~S_n = n-1 \right),
    \]
    where
    \(\eql\) denotes ``identically distributed''
    and
    the right-hand-side denotes \( (\xilist)\) restricted to the event that \(S_n = n-1\).
\end{lemma}

Let \(p_i := \p{\xi = i}\).
Let \(h\) be the \emph{span} of \(\xi\), i.e.,
\(
    h := {\gcd}\{i \ge 1: p_i > 0\}.
\)
We recall the following result (see \citet[(4.3)]{Janson2014Asym} or \citet{Kolchin1986random}):
\begin{lemma}
    \label{pre:lem:prob:T:eql:n}
                Assume Condition \ref{cond:offspring}.
    We have
    \[
        \p{|\gw|=n} \sim \frac{h}{\sqrt{2 \pi \sigma^2}} n^{-3/2},
    \]
    as \(n \to \infty\) with \(n-1 \equiv 0 \pmod{h}\).
\end{lemma}

The following lemma is a special case of \cite[lem.\ 5.1]{Janson2014Asym}.
We nonetheless give a short proof for later reference in the paper.
\begin{lemma}
    \label{pre:lem:tree:count:expectation}
    Assume \(\p{|\gw| = n}>0\).
    Let \(T \in \cS_k\) with \(1 \le k \le n\).
    \begin{enumerate}[label=(\roman*)]
        \item Let \(\treecount_T(\gwn) := \sum_{v \in \gwn}\iverson{\gw_{n,v} = T}\).  Then
        \[
            \frac{\E{\treecount_T(\gwn)}}{n}  =  \pi(T) \frac{\p{S_{n-k} = n-k}}{\p{S_n = n-1}}.
        \]
        \item Let \(\treecount_{\cS_k}(\gwn) = \sum_{v \in \gwn}\iverson{|\gw_{n,v}| = k}\). Then
        \[
            \frac{\E{\treecount_{\cS_k}(\gwn)}}{n} =  \pi(\cS_k) \frac{\p{S_{n-k} = n-k}}{\p{S_n = n-1}}.
        \]
    \end{enumerate}
\end{lemma}
\begin{proof}
    Let \( (d_1, d_2, \ldots{}, d_k) \) be the preorder degree sequence of \(T\).
            Recall that \( (\xinlist)\) is the preorder degree sequence of \(\gwn\).
    Let
    \[
        I_{i} =
        \iverson{\xin_i = d_1, \xin_{i+1} = d_2, \ldots, \xin_{i+k-1} = d_{k}},
    \]
    where the indices are taken modulo \(n\).

    Note that if \(n-k + 1< i \le n\), then it is impossible that \(I_{i} = 1\), because the
    length of the preorder degree sequence of the fringe subtree \(\gw_{n, v_i}\) must be strictly
    less than \(k\). Therefore, if \(I_{i} > 1\) and \(n-k+1< i \le n\), then there exists a \(k' <
    k\) such that \( (d_1, d_2,\ldots, d_k')\) is also a preorder degree sequence, which is impossible by
    Corollary \ref{cor:pre:tree:deg:seq}.

    Therefore for all \(1 \le i \le n\),
    \(I_{i}= \iverson{T_{v_i} = T} \) and \(\treecount_{T}(\gwn) = \sum_{i=1}^{n} I_i\).
    Recalling that \( (\xitildenlist) \) is a uniform random rotation of \( (\xinlist)
    \) and using Lemma \ref{pre:lem:tree:deg:seq:rotation}, we have
    \begin{align*}
        \E{\treecount_{T}(\gwn)}
        &
        =
        \E{\sum_{i=1}^{n} I_i}
        =
        \sum_{i=1}^{n} \p{\xin_i = d_1, \xin_{i+1} = d_2, \ldots, \xin_{i+k-1} = d_{k}}
                                \\
        & =
        \sum_{i=1}^{n} \p{\xi_i = d_1, \xi_{i+1} = d_2, \ldots, \xi_{i+k-1} = d_{k}~|~S_n = n-1}
        \\
        & =
        n \p{\xi_1 = d_1, \xi_{2} = d_2, \ldots, \xi_{k} = d_{k}~|~S_n = n-1}
        \\
        & =
        n
        \frac{
            \p{
                \left[\xi_1 = d_1, \xi_{2} = d_2, \ldots, \xi_{k} = d_{k} \right]
                \cap
                \left[S_{n} = n-1 \right]
            }
        }
        {
            \p{S_{n} = n-1}
        }
        .
    \end{align*}
    Since \( (d_1, \ldots d_k) \) is a preorder degree sequence, \(\sum_{i=1}^{k} d_i = k-1\) by Lemma
    \ref{pre:lem:tree:deg:seq}.
    Therefore, using the fact that \( \xilist\) are independent, the last expression equals
    \begin{align*}
        n \frac{
            \p{
                \left[\xi_1 = d_1, \xi_{2} = d_2, \ldots, \xi_{k} = d_{k} \right]
                \cap
                \left[S_{n-k} = n-k \right]
            }
        }
        {
            \p{S_{n} = n-1}
        }
        =
        n \p{\gw = T} \frac{\p{S_{n-k} = n-k}}{\p{S_n = n-1}}
        .
    \end{align*}
    Thus part (i) is proved. Part (ii) follows by summing the equality in (i) over all \(T \in
    \cS_k\).
\end{proof}

The following approximations are useful for estimating the expectation and the variance of the
number of fringe subtrees.
\begin{lemma}[Lemma 5.2 and 6.2 of \citet{Janson2014Asym}]
    \label{pre:lem:prob:sum:of:iid}
    Assume Condition \ref{cond:offspring}.
    We have: \\
    (i) Uniformly for all \(k\) with \(1 \le k \le n/2\),
    \begin{align*}
        \frac{\p{S_{n-k} = n-k}}{\p{S_{n} = n-1}} = 1
        + \bigO{\frac{k}{n}}
        + o\left( n^{-1/2} \right).
    \end{align*}
    (ii) Uniformly for all \(k\) with \(1 \le k \le n/4\),
    \begin{align*}
        \frac{\p{S_{n-2k}=n-2k+1}}{\p{S_n=n-1}}
        -
        \left(
        \frac{\p{S_{n-k}=n-k}}{\p{S_{n}=n-1}}
        \right)^{2}
        =
        O\left( \frac{1}{n} \right)
        + O\left( \frac{k}{n^{3/2}} \right)
        + O\left( \frac{k^{2}}{n^2} \right)
        .
    \end{align*}
\end{lemma}

\subsection{Poisson Approximation}

Stein's method \citep{Stein1972} bounds the errors in approximating the
sum of random variables with a normal distribution. It was extended to Poisson distributions first
by \citet{Chen1975} and later by many others, e.g., \citet{Arratia1989} and
\citet{Barbour1992poisson}.  Among these, we found that it is easier to use the
\emph{exchangeable pairs} method by \citet{chatterjee2005} to prove Theorem \ref{intro:thm:fringe:count}.  In
particular, we use it of the form of \citet[thm.\ 4.37]{ross2011fundamentals} as follows.

Given a function \(f:\N_0\to\R\), let
\[
    \norm{f} := \max_{i \in \N_0} |f(i)|,
    \qquad
    \norm{\Delta f} := \max_{i \in \N_0} |f(i+1)-f(i)|
    .
\]
\begin{lemma}
    \label{pre:lem:exchangeable:pairs}
    Let \(W\) be a non-negative integer-valued random variable with \(\e W = \lambda\).
    Let \( (W, W') \) be an exchangeable pair, i.e., \((W,W') \eql (W',W)\).
    Let \(\cF\) be a \(\sigma\)-algebra such that \(W\) is \(\cF\)-measurable.
    If \(Z \eql \Po(\lambda)\), then there exists a function \(f:\N_0 \to \R\) with
    \[
        \norm{f} \le \lambda^{-1/2}, \qquad \norm{\Delta f} \le \lambda^{-1},
    \]
    such that for all \(c \in \R\),
    \begin{align*}
        \distTV{W,Z} \le
        \E{
            \left(
            \lambda-c \p{W'=W+1 |\cF}
            \right)
            f(W+1)
        }
        \\
        -
        \E{
            \left(
            W
            -
            c \p{W'=W-1 |\cF}
            \right)
            f(W)
        }
        .
    \end{align*}
\end{lemma}

The following Lemma is a special case of \citet[thm.\ 1]{roos2003467}, which applies to mixed
Poisson distributions.  \citeauthor{Barbour1992poisson} proved a similar result using Stein's method \cite[thm.\
1.C]{Barbour1992poisson}.  We include our proof for its simplicity.
\begin{lemma}
    \label{pre:lem:poisson:tv:distance}
    If \(X \eql \Po(\mu)\) and \(Y \eql \Po(\nu)\), then
    \[
        \distTV{X, Y} \le \left|\sqrt\mu-\sqrt{\nu}\right|
        = \frac{|\mu - \nu|}{\sqrt{\mu} + \sqrt{\nu}}.
    \]
\end{lemma}

\begin{proof}
    Let \(x_i = \p{X=i}\) and \(y_i = \p{Y=i}\). We have
    \begin{align*}
        \distTV{X,Y}
        &
        =
        \frac{1}{2}
        \sum_{i=1}^{\infty} |x_i - y_i|
        =
        \frac{1}{2}
        \sum_{i=1}^{\infty} |\sqrt{x_i} - \sqrt{y_i}|(\sqrt{x_i}+\sqrt{y_i}) \\
        & \le
        \frac{1}{2}
        \left(
        \sum_{i=1}^{\infty} |\sqrt{x_i} - \sqrt{y_i}|^2
        \sum_{j=1}^{\infty} (\sqrt{x_j}+\sqrt{y_j})^2
        \right)^{1/2}
        \\
        & =
        \frac{1}{2}
        \left(
        \left(
        2-
        \sum_{i=1}^{\infty} 2 \sqrt{x_iy_i}
        \right)
        \left(
        2+
        \sum_{j=1}^{\infty} 2 \sqrt{x_jy_j}
        \right)
        \right)^{1/2}
        \\
        &=
        \left( 1-\left( \sum_{i=1}^{\infty} \sqrt{x_i y_i} \right)^2 \right)^{1/2},
    \end{align*}
    where the second step uses the Cauchy-Schwartz inequality.
    An easy calculation shows that
    \[
        \sum_{i=1}^{\infty} \sqrt{x_i y_i}
        =
        \sum_{i=1}^{\infty} e^{-\frac{\mu+\nu}{2}} \frac{(\mu\nu)^{i/2}}{i!}
        =
        \exp \left( \sqrt{\mu\nu}-\frac{\mu+\nu}{2} \right)
        =
        \exp \left( -\frac{(\sqrt{\mu}-\sqrt{\nu})^2}{2} \right).
    \]
    Thus we have
    \begin{align*}
        \distTV{X,Y}
        &
        \le
        \sqrt{1-\exp\left( -(\sqrt{\mu}-\sqrt{\nu})^2 \right)}
        \\
        &
        \le \sqrt{(\sqrt{\mu}-\sqrt{\nu})^2}
        \qquad (\text{by \(1-e^{-x} \le x\)})
        \\
        &
        =
        \left|\sqrt{\mu}-\sqrt{\nu} \right|
        .
        \qedhere
    \end{align*}
\end{proof}

\section{Sequences of fringe subtrees}

\label{sec:fringe}

In this section we prove Theorem \ref{intro:thm:fringe:count} and sketch the proof of
Theorem \ref{intro:thm:fringe:set:count}.
Recall that \(\pi(T) := \p{\gw = T}\), and \(\treecount_{T}(n) := \iverson{\sum_{v \in \gwn} \gw_{n,v} = T}\).
Theorem \ref{intro:thm:fringe:count} states that
\[
    \sup_{T \in \tsetkn} \distTV{\fringeT,\Po(n \pi(T))} = o(1),
\]
whenever \(k_n = o(n)\).
If \(\pi(T) = 0\), then \(\fringeT = 0\) and \( \distTV{\fringeT, \Po(n \pi(T))}=0\)
deterministically.
Thus we can assume that \(\p{\gw \in \tsetkn} = \p{|\gw| = k_n} > 0\) for all \(n\),
and that the above supremum is taken over all \(T \in \tsetkn\) with \(\pi(T) > 0\).

To prove this, we first show that \(\e \fringeT \sim n \pi(T)\).
Then we construct a coupling \(\bfringeT \sim \fringeT\)
such that \( (\fringeT, \bfringeT)\) forms an exchangeable pair,
and apply Lemma \ref{pre:lem:exchangeable:pairs} to finish the proof.

Recall that \(\fringeSkn := \iverson{\sum_{v \in \gwn} T_v \in \tsetkn}\), i.e.,
\(\fringeSkn\) is the number of fringe subtrees of size \(k_n\) in \(\gwn\).
Also recall that \(\pi(\cS) := \p{\gw \in \cS}\).

\begin{lemma}
    \label{fringe:lem:tree:count:expectation}
    Let \(k = k_n = o(n)\).
    We have
    \begin{align*}
        \sup_{T \in \tsetk}
        \left|
        \frac{\e \fringeT}{n\pi(T)}
        -1
        \right|
        =
        O\left( \frac{k}{n} \right) +o \left( n^{-1/2} \right)
        ,
    \end{align*}
    and
    \begin{align*}
        \left|
        \frac{\e \fringeSk}{n \pi(\tsetk)}
        -1
        \right|
        =
        O\left( \frac{k}{n} \right) +o \left( n^{-1/2} \right)
        .
    \end{align*}
\end{lemma}

\begin{proof}
    Since \(k = o(n)\), we have \(k < n/2\) for \(n\) large.
    Thus by Lemma \ref{pre:lem:tree:count:expectation} and \ref{pre:lem:prob:sum:of:iid},
    uniformly for all \(T \in \tsetk\),
    \begin{align*}
        \left|
        \frac{\e \fringeT}{n}
        -
        \pi(T)
        \right|
        &
        =
        \pi(T)
        \left|
        \frac{\p{S_{n-k} = n- k}}{\p{S_n = n-1}}
        -
        1
        \right|
        =
        \pi(T)
        \left(
        \bigO{ \frac{k}{n}} + o(n^{-1/2})
        \right)
        .
    \end{align*}
    Summing over all trees \(T\) with \(|T|=k\) gives the second part of the lemma.
\end{proof}

\begin{lemma}
    \label{fringe:lem:tree:count:var}
    Let \(k = k_n = o(n)\).
    There exists a constant \(C_1\) such that
    \[
    \V {\fringeSk}
    \le C_1 (n \pi(\tsetk))^2
        \left(
        \frac{1}{n}
        +
        \frac{k}{n^{3/2}}
        +
        \frac{k^{2}}{n^2} \right)
        .
    \]
\end{lemma}

\begin{proof}
    Recall that \( \xinvec := (\xinlist)\) is the preorder degree sequence of \(\gwn\)
    and that \(\cD_{k}\) is the set all preorder degree sequences of length \(k\).
    Let
    \[
    J_i :=
            \iverson{(\xin_{i}, \xin_{i+1}, \ldots \xin_{i+k-1}) \in \cD_k}
    ,
    \]
    where the indices are modulo \(n\).
    Thus \(J_i = 1\) if and only if \((\xin_{i}, \ldots \xin_{i+k-1})\)
    satisfies \eqref{pre:eq:deg:seq}.
    Note that if \(n-k+1 < i \le k\), then it is impossible that \(J_i = 1\).
    (See Lemma \ref{pre:lem:tree:count:expectation} for a similar argument.)
            Thus
    \(\fringeSk := \sum_{v \in \gwn} \iverson{\gw_{n,v} \in \tsetk} = \sum_{i=1}^{n} J_i\).

    Recall that \( \xitildenvec:=(\xitildenlist) \) is a uniform random rotation of \( \xinvec\).
    Let
    \[
        \Jtilde_i :=
        \iverson{(\xitilden_{i}, \xitilden_{i+1}, \ldots \xitilden_{i+k-1}) \in \cD_k}.
    \]
    Then \(\fringeSk = \sum_{i=1}^{n} J_i = \sum_{i=1}^{n} \Jtilde_i\).

    Using the fact that \(\Jtilde_1, \ldots, \Jtilde_n\) are identically distributed
    and Lemma \ref{pre:lem:tree:count:expectation}, we have
    \[
        \e \Jtilde_1 = \frac{1}{n} \e \fringeSk
        =
        \p{|\gw| = k} \frac{\p{S_{n-k} = n-k}}{\p{S_n = n-1}}.
    \]
    Similar to the proof of Lemma \ref{pre:lem:tree:count:expectation}, we have
    \begin{align*}
        \e \Jtilde_1 \Jtilde_{k+1}
        & =
        \p{
            (\xitilden_{1}, \ldots, \xitilden_{k}) \in \cD_{k},
            (\xitilden_{k+1}, \ldots, \xitilden_{2k}) \in \cD_{k}
        }
        \\
        & =
        \p{
            (\xi_{1}, \ldots, \xi_{k}) \in \cD_{k},
            (\xi_{k+1}, \ldots, \xi_{2k}) \in \cD_{k}
            ~|~S_{n} = n-1
        }
        \\
        & =
        \p{
            |\gw| = k
        }^{2}
        \frac{
            \p{S_{n-2k} = n-2k+1}
        }
        {
            \p{S_{n} = n-1},
        }
    \end{align*}
    where \(\xi_1, \xi_2, \ldots\) are i.i.d.\ copies of \(\xi\) and \(S_{m} := \sum_{i=1}^{m} \xi_i\).

    Consider two indices \(i \ne j\).
    Note that if \(|i-j| < k\) or \(|i+n-j| < k\), then
    \(\e \Jtilde_i \Jtilde_j = 0\). This is because a fringe subtree of size \(k\), besides itself, cannot contain another
    fringe subtree of the same size.
    On the other hand, if \(|i-j| > k\) and \(|i+n-j| > k\),
    i.e.,
    \( (\xitilden_i, \ldots, \xitilden_{i+k-1}) \) and
    \( (\xitilden_j, \ldots, \xitilden_{j+k-1}) \) do not overlap,
    then \(\E{\Jtilde_i \Jtilde_j} = \E{\Jtilde_1
        \Jtilde_{k+1}}\) since \( \xitildenvec\) is permutation invariant.
        Therefore, there exists a constant \(C_1\) such that
    \begin{align*}
        \V{\fringeSk}
        &
        =
        \sum_{1 \le i, j \le n}
        \left(
        \E{\Jtilde_i \Jtilde_j}
        -
        \E{\Jtilde_i}\E{\Jtilde_j}
        \right)
        \le
        n(n-2k)
        \left(
            \e \Jtilde_1 \Jtilde_{k + 1}
            - \e{\Jtilde_1}\e{\Jtilde_{k+1}}
        \right)
        \\
        &
        \le
        \left(n \p{|\gw| = k} \right)^{2}
        \left[
            \frac{\p{S_{n-2k} = n-2k+1}}{\p{S_n = n-1}}
            -
            \left(
            \frac{\p{S_{n-k} = n-k}}{\p{S_n = n-1}}
            \right)^2
        \right]
        \\
        &
        =
        C_1
        \left(n  \pi(\tsetk)  \right)^{2}
        \left(
        \frac{1}{n}
        + \frac{k}{n^{3/2}}
        + \frac{k^{2}}{n^2}
        \right)
        ,
    \end{align*}
    where in the last step we use
    \(k \le n/4\) for \(n\) large
    and part (ii) of Lemma \ref{pre:lem:prob:sum:of:iid}.
\end{proof}

The following Lemma is the key part of the proof of Theorem \ref{intro:thm:fringe:count}.
\begin{lemma}
    \label{fringe:lem:tree:count:distance}
                Assume that \(k=k_n=o(n)\) and \(k \to \infty\).
    We have as \(n \to \infty\),
    \[\sup_{T \in \tsetk} \distTV{\fringeT, \Po(\E{\fringeT})} \to 0.\]
\end{lemma}

\begin{proof}
    Let \(T \in \tsetk\).
    We use the abbreviations
    \(\Gamma_n = \fringeT\) and \(\Psi_n = \fringeSk\).
    Let \( (d_1, \ldots, d_k) \) be the preorder degree sequence of \(T\).
    Recall that \( \xitildenvec := (\xitildenlist)\) is a uniform random cyclic rotation of
    the preorder degree sequence of \(\gwn\).
    Let
    \[
        \Itilde_{i} = \iverson{\xitilden_{i} = d_1, \xitilden_{i+1} = d_2, \ldots \xitilden_{i+k-1} = d_k},
        \quad
        \text{ and }
        \quad
        \Jtilde_{i} = \iverson{(\xitilden_i, \xitilden_{i+1}, \ldots, \xitilden_{i+k-1}) \in \cD_k},
    \]
    where all the indices are modulo \(n\).
    Then by the same argument in Lemma \ref{fringe:lem:tree:count:var},
    \(\Gamma_n = \sum_{i=1}^{n} \Itilde_{i}\)
    and \(\Psi_n = \sum_{i=1}^{n} \Jtilde_{i}\).
    Note that \(\Itilde_{i} = 1\) implies that
    \(\Jtilde_{i} = 1\) since \( (d_{1}, \ldots, d_{k}) \in \cD_{k}\).
    Let \( \xihatnvec := (\xihatnlist)\) be an independent copy \( \xitildenvec \).
    Let \(\Ihat_i\) and \(\Jhat_i\) be defined for \( \xihatnvec \)
    as \(\Itilde_{i}\) and \(\Jtilde_{i}\) for \(\xitilden\).

    We construct a new sequence \( \xiovernvec := (\xiovernlist)\) as follows.
    Let \(Z\) be a uniform random variable on \(\{1,2,\ldots,n\}\) that is independent from everything else.
    Let \( \xiovernvec = \xitildenvec \) if \(\Jtilde_{Z}=0\) or \(\Jhat_{Z} = 0\).
    Otherwise \(\Jtilde_{Z}=\Jhat_{Z} = 1\) and in this case we let
    \[
        \xiovern_{i} =
        \begin{cases}
            \xihatn_i & \qquad \mbox{if } i \in \{Z, Z+1, \ldots, Z+k-1\} \pmod{n}, \\
            \xitilden_i &  \qquad \mbox{otherwise}.
        \end{cases}
    \]

    Equivalently, we can construct \( \xiovernvec\) using a method which we call \emph{subtree
    switching}.
    Let \(\tildegwn\) be an independent copy of \(\gwn\).
    Let \(Z\) be as before.
    Let \(v({Z})\) and \(\widetilde{v}({Z})\) be the \(Z\)-th nodes (in {\scshape dfs} order) in \(\gwn\) and
    \(\tildegwn\) respectively.
    If the two fringe subtrees \(\gw_{n, v(Z)}\) and \(\tildegw_{n, \widetilde{v}({Z})}\) are both of size
    \(k\), then in \(\gwn\) we replace \(\gw_{n, v(Z)}\) with \(\gw_{n, \widetilde{v}({Z})}\) and get a new tree
    \(\overgwn\). Otherwise let \(\overgwn = \gwn\).
    Let \( \xiovernvec\) be a uniform random rotation of the preorder degree sequence of \(\overgwn\).
    This construction of \( \xiovernvec\) is equivalent to the above one.

    Let \(\Iover_{i}\) be defined for \( \xiovernvec\) as \(\Itilde_i\) for \( \xitildenvec\).
    Let \(\Gammaover_n := \sum_{i=1}^{n} \Iover_i = \bfringeTn\).
    It is not hard to verify that \( \xitildenvec \eql \xiovernvec\) and that
    \( (\Gamma_n, \Gammaover_n) \eql (\Gammaover_n, \Gamma_n)\),
    i.e., \( (\Gamma_n, \Gammaover_n) \) is an exchangeable pair defined in Lemma
    \ref{pre:lem:exchangeable:pairs}. We leave the details to the reader.

    Let \(r_n := \e{\Itilde_1} = \e{\Ihat_1}\)
    and \(q_n := \e{\Jtilde_1} = \e{\Jhat_1}\).
    Let \(\cF = \sigma(\xitildenlist)\), i.e., the sigma algebra generated by \( \xitildenvec\).
    Since the event \(\Gammaover_n = \Gamma_n + 1\) happens if and only if \(\Jtilde_{Z} =
    \Jhat_{Z} = 1\), \(\Itilde_{Z} = 0\) and \(\Ihat_{Z} = 1\), we have
    \begin{align*}
        \p{\Gammaover_n = \Gamma_n + 1|\cF}
        &
        =
        \p{ \cup_{i=1}^{n}
        \left.
        \left[
            Z=i, \Jtilde_{i} = 1, \Itilde_{i}=0, \Jhat_{i} = 1, \Ihat_{i} = 1
        \right]\,
        \right| \cF}
        \\
        &
        =
        \sum_{i=1}^{n}
        \p{Z = i}
        \p{\left. \Jtilde_{i}=1, \Itilde_{i}=0 \, \right|\cF}
        \p{\Ihat_{i}=1}
        \\
        &
        =
        \frac{1}{n}
        \sum_{i=1}^{n}
        \Jtilde_{i}
        \left( 1-\Itilde_i \right)
        r_{n}
        =
        \frac{r_n}{n}
        \sum_{i=1}^{n}
        \left(\Jtilde_{i} -\Itilde_i \right),
    \end{align*}
    where we use that
    \([\Itilde_i = 1] \subseteq [\Jtilde_i = 1] \)
    and
    \([\Ihat_i = 1] \subseteq [\Jhat_i = 1] \).

    Let \(c =  n/\e \Jtilde_1=n/q_n \).
    Note that \(\sum_{i=1}^{n} \e \Itilde_i = \e \Gamma_n = n r_n\),
    and that \(\sum_{i=1}^n \e \Jtilde_i = \e \Psi_n = n q_n\).
    Thus
    \begin{align*}
        \zeta_1
        & :=
        \e \Gamma_n - c \p{\Gammaover_n = \Gamma_n + 1|\cF}
        =
        \e \Gamma_n - \frac{n}{q_n} \frac{r_n}{n} \sum_{i=1}^{n}
        \left(
            \Jtilde_i
            -\E{\Jtilde_i}
            +\E{\Jtilde_i}
            -\Itilde_i
        \right)
        \\
        &
        =
        n r_n - \frac{r_n}{q_n} n q_n
        + \frac{r_n}{q_n} \Gamma_n
        - \frac{r_n}{q_n}
        \sum_{i=1}^n
        \left( \Jtilde_i - \e \Jtilde_i \right)
        =
        \frac{r_n}{q_n} \Gamma_n
        + \frac{r_n}{q_n}
        \left(\E{\Psi_n} - \Psi_n \right)
        .
    \end{align*}

    The event \(\Gammaover_n = \Gamma_n - 1\) happens if and only if \(\Jtilde_{Z} =
    \Jhat_{Z} = 1\), \(\Itilde_{Z} = 1\) and \(\Ihat_{Z} = 0\). Therefore
    \begin{align*}
        \p{\Gammaover_n = \Gamma_n - 1|\cF}
        &
        =
        \p{ \cup_{i=1}^{n}
        \left.
        \left[
            Z=i,
            \Jtilde_{i} = 1,
            \Itilde_{i}=1,
            \Jhat_{i} = 1,
            \Ihat_{i} = 0
        \right]\,
        \right| \cF}
        \\
        &
        =
        \sum_{i=1}^{n}
        \p{Z = i}
        \p{\left. \Jtilde_{i}=1, \Itilde_{i}=1 \, \right|\cF}
        \p{\Jhat_{i}=1, \Ihat_{i} = 0}
                                                                        \\
        &
        =
        \frac{1}{n}
        \sum_{i=1}^n
        \Itilde_i \left(
        \p{\Jhat_i = 1}
        -
        \p{\Ihat_i = 1}
        \right)
        =
        \frac{q_n}{n} \Gamma_n
        -
        \frac{r_n}{n} \Gamma_n
        ,
    \end{align*}
    where we again use that
    \([\Itilde_i = 1] \subseteq [\Jtilde_i = 1] \)
    and
    \([\Ihat_i = 1] \subseteq [\Jhat_i = 1] \).
    Therefore
    \begin{align*}
        \zeta_2 :=
        \Gamma_n - c \p{\Gammaover_n = \Gamma_n -1 | \cF}
        =
        \Gamma_n - \frac{n}{q_n} \left(
        \frac{q_n}{n} \Gamma_n
        -
        \frac{r_n}{n} \Gamma_n
        \right)
        =
        \frac{r_n}{q_n} \Gamma_n.
    \end{align*}

        It follows from Lemma \ref{pre:lem:exchangeable:pairs} that
    there exists a function \(f:\N_0 \to \R\) with
    \(\norm{\Delta f} \le \E{\Gamma_n}^{-1} = (n r_n)^{-1}\)
    and
    \(\norm{f} \le \E{\Gamma_n}^{-1/2} = (n r_n)^{-1/2}\), such that
    \begin{align*}
        \distTV{\Gamma_n,\Po(\e \Gamma_n)}
        &
        \le
        \E{
            \zeta_1 f(\Gamma_n+1)
        -
            \zeta_2 f(\Gamma_n)
        }
        \\
        &
        =
        \E {
            \left(
                \frac{r_n}{q_n} \Gamma_n + \frac{r_n}{q_n}
                \left(
                    \E{\Psi_n}-\Psi_n
                \right)
            \right)
            f(\Gamma_n+1)
            -
            \frac{r_n}{q_n}
            \Gamma_n
            f(\Gamma_n)
        }
        \\
        &
        =
        \E{
            \frac{r_n}{q_n} \Gamma_n \left( f(\Gamma_n+1) - f(\Gamma_n)\right)
            +
            f(\Gamma_n+1) \frac{r_n}{q_n}
                \left(
                    \E{\Psi_n}-\Psi_n
                \right)
        }
        \\
        &
        \le
        \frac{r_n}{q_n}
        \E{
            \Gamma_n \cdot \left| f(\Gamma_n+1) - f(\Gamma_n)\right|
        }
        +
        \frac{r_n}{q_n}
        \E{
            |f(\Gamma_n+1)| \cdot
                \left|
                    \E{\Psi_n}-\Psi_n
                \right|
        }
        \\
        &
        \le
        \frac{r_n}{q_n}
        \norm{\Delta f}
        \E{
            \Gamma_n
        }
        +
        \frac{r_n}{q_n}
        \norm{f}
        \E{
            \left| \E{\Psi_n}-\Psi_n \right|
        }
        \\
        &
        \le
        \frac{r_n}{q_n}
        \frac{1}{n r_n} n r_n
        +
        \frac{r_n}{q_n}
        \frac{1}{\sqrt{n r_n}}
        \left(\V{\Psi_n} \right)^{1/2}
        \\
        &
        =
        \frac{r_n}{q_n}
        +
        \left[
            \frac{r_n \V{\Psi_n}}{n q_n^2}
        \right]^{1/2}
        := \alpha_1 + \alpha_2
        .
    \end{align*}

    Let \(\pmax := \max_{i \ge 0} p_i\).
    By the assumptions of Theorem \ref{intro:thm:fringe:count}, \(\V{\xi} > 0\), which implies that \(\xi\) is
    not a constant. Thus \(\pmax < 1\).
    By Lemma \ref{fringe:lem:tree:count:expectation}, we have
    \[
        r_n
        =
        \frac{\e \Gamma_{n}}{n}
        \sim
        \pi(T_n)
        :=
        \p{\gw = T_n}
        =
        p_{d_1} p_{d_2} \cdots p_{d_{k}}
        \le \pmax^{k}.
    \]
    And by Lemma \ref{fringe:lem:tree:count:expectation} and Lemma \ref{pre:lem:prob:T:eql:n}
    \[q_n
        = \frac{\e \Psi_{n}}{n}
        \sim \pi(\tsetkn) := \p{|\gw|= k} = \Theta(k^{-3/2}).\]
    Therefore \(\alpha_1 = r_n/q_n = O(\pmax^{k} /k^{-3/2}) = o(1)\).

    By Lemma \ref{fringe:lem:tree:count:var}, we have
    \begin{align*}
        \alpha_2^{2}
        =
        \frac{r_n \V{\Psi_n}}{n q_n^2}
        &
        =
        \frac{r_n}{n(1+o(1)) \pi(\tsetkn)^2} (n \pi(\tsetkn))^2
        \left(
        O\left( \frac{1}{n} \right)
        +
        O\left( \frac{k}{n^{3/2}} \right)
        +
        O\left( \frac{k^2}{n^2} \right)
        \right)
        \\
        &
        =
        O(r_n)
        +
        O\left(\frac{r_n k}{n^{1/2}} \right)
        +
        O\left( \frac{r_n k^{2}}{n} \right)
        \to 0,
    \end{align*}
    since
    \(r_n \le \pmax^{k} \to 0\),
    \(r_n k \le \pmax^{k} k \to 0 \) and \(k = o(n)\).
    Therefore \(\alpha_2 \to 0\).
    Thus \(\distTV{\Gamma_n, \Po(\e \Gamma_n)} \le \alpha_1 + \alpha_2 \to 0\).
    Since this upper bound holds whenever \(T \in \tsetk\), the lemma follows.
\end{proof}

\begin{proof}[Proof of Theorem \ref{intro:thm:fringe:count}]
    Let \(k = k_n\).
    By Lemma \ref{fringe:lem:tree:count:expectation},
    \begin{align*}
        \sup_{T \in \tsetk}
        \left|
        \frac{\e \fringeT}{n \pi(T)} - 1
        \right|
        =
        O\left( \frac{k}{n} \right) +o \left( n^{-1/2} \right).
    \end{align*}
    Therefore, for all \(T \in \tsetk\), we have
    \begin{align*}
        \frac
        {
            \left|
            n \pi(T) - \e \fringeT
            \right|
        }
        {
            \sqrt{n \pi(T)}
        }
        &
        =
        \sqrt{n \pi(T)}
        \left(O\left( \frac{k}{n} \right) +o \left( n^{-1/2} \right) \right)
        \\
        &
        =
        O\left(
        \sqrt{ \frac{k^2 \pi(T)}{n}}
        \right)
        + o(\pi(T))
        =
        \smallo{\sqrt{k \pmax^{k}}}
        ,
    \end{align*}
    where the last step uses \(k = o(n)\) and \(\pi(T) \le \pmax^{k}\).
    It follows from Lemma \ref{pre:lem:poisson:tv:distance} that
    \[
        \sup_{T \in \tsetk}
        \distTV{\Po(n \pi(T)), \Po(\e \fringeT)}
        \le
        \bigO{
        \sup_{T \in \tsetk}
        \frac
        {
            \left|
            n \pi(T) - \e \fringeT
            \right|
        }
        {
            \sqrt{n \pi(T)}
        }
        }
        =
        \smallo{\sqrt{k \pmax^{k}}}
        .
    \]

    Thus by Lemma \ref{fringe:lem:tree:count:distance}
    and the triangle inequality,
    \begin{align*}
        &
        \sup_{T \in \tsetk}
        \distTV{\fringeT, \Po(n \pi(T))}
        \\
        &
        \le
        \sup_{T \in \tsetk}
        \distTV{\fringeT, \Po(\e \fringeT)}
        +
        \sup_{T \in \tsetk}
        \distTV{\Po(n \pi(T)), \Po(\e \fringeT)}
        \\
        &
        = o(1) + \smallo{\sqrt{k \pmax^{k}}}
        = o(1).
    \end{align*}
    Statements (i)--(iii) of theorem follow by Lemma \ref{intro:lem:poisson:distribution}.
\end{proof}

\subsection{The sketch of the proof of Theorem \ref{intro:thm:fringe:set:count}}

Let \(\cA_n \subseteq \tsetkn\) be a sequence of sets of trees, i.e., \(\cA_n\) is a set of some
trees of size \(k_n\).
Theorem \ref{intro:thm:fringe:set:count} states that
\(\sup_{\cA_n \subseteq \tsetkn} \distTV{\treecount_{\cA_n}(\gwn), \pi(\cA_n)}
\to 0\).
We can show this by adapting the proof of Theorem \ref{intro:thm:fringe:count}.
The key step is to construct an exchangeable pair \(
(\treecount_{\cA_n}(\gwn),\treecount_{\cA_n}(\overgwn)) \)
as in the proof of Lemma \ref{fringe:lem:tree:count:distance}
with new definitions of \(\Itilde_i, \Jtilde_i, \Ihat_i\) and \(\Jhat_i\).

Let \(\cD(\cA_n)\) be the set of preorder degree sequences of trees in \(\cA_n\).
Recall that \( \xitildenvec := (\xitildenlist)\) is a uniform random rotation of the preorder degree
sequence of \(\gwn\).
Let
\[
    \Itilde_i = \iverson{ (\xitilden_i, \xitilden_{i+1},\ldots, \xitilden_{i+k-1}) \in \cD(\tsetkn)},
    \qquad
    \Jtilde_i = \iverson{ \xitilden_i + \xitilden_{i+1}\cdots + \xitilden_{i+k-1} = k-1}.
\]
Then \(\treecount_{\cA_n}(\gwn)=\sum_{i=1}^{n} \Itilde_i\).
Let \( \xihatnvec:=(\xihatnlist) \) be an independent copy of \( \xitildenvec\).
Define \(\Ihat_i\) and \(\Jhat_i\) for \( \xihatnvec\) as \( \Itilde_i\) and \(\Jtilde_i\) for
\( \xitildenvec\).
Now we construct \( \xiovernvec \eql \xitildenvec\) as in Lemma \ref{fringe:lem:tree:count:distance},
which gives a random tree \(\overgwn \eql \gwn\).
Then \( (\fringeTn, \bfringeTn)\) is an exchangeable pair.
From here we can argue as in the rest of the proof for Theorem
\ref{intro:thm:fringe:count}.

\section{Families of fringe subtrees}

\label{sec:fringe:app}

In this section, we apply Theorem \ref{intro:thm:fringe:count} and \ref{intro:thm:fringe:set:count}
to study the conditions for \(\gwn\) to contain every tree that belongs to a family of trees.

\subsection{Coupon collector problem}

\label{sec:fringe:app:coupon}

As shown later, our problem is essentially a variation of the famous coupon collector
problem---if in every draw we get a coupon with a uniform random type among \(n\) types, how many
draws do we need to collect all \(n\) types of coupons?  The next lemma is about a generalization of this
problem needed later.
For the original problem, see \citet{erdos61coupon} and \citet{flajolet1992}.
For more about the generalized version defined below, see \citet{neal2008}.

\begin{lemma}[Generalized coupon collector]
    \label{family:lem:coupon}
    Let \(X_n\) be a random variable that takes values in \(\{1, \ldots, n\}\).
    Let \(p_{n,i} := \p{X_n = i }\).
    Assume that \(p_{n,i} > 0\) for all \(1 \le i \le n\).
    Let \(X_{n,1}, X_{n,2}, \ldots\) be i.i.d.\ copies of \(X_n\).
    Let
    \[
        N_n := \inf\{i \ge 1:|\{X_{n,1}, X_{n,2}, \ldots, X_{n,i}\}| = n\}.
    \]
    Let \(m_n\) be a sequence of real numbers.
    We have
    \[
        1-\sum_{i=1}^n (1-p_{n,i})^{m_n}
        \le
        \p{N_n \le m_n}
        \le
        \frac{1}{\sum_{i=1}^n (1-p_{n,i})^{m_n}}.
    \]
\end{lemma}

\begin{proof}
    Let \(m = m_n\).
    Let \(Z_{n,i} = \iverson{i \notin \{X_{n,1},\ldots, X_{n, m}\}}\).
    Then \(N_n \le m\) if and only if \(Z_n := \sum_{i=1}^n Z_{n,i} = 0\), i.e.,
    \(\p{N_n \le m} = \p{Z_n = 0}=1-\p{Z_n \ge 1}\).

    The first inequality of this lemma follows from the following:
    \[
        \p{Z_n \ge 1}
        \le \e Z_n
        = \sum_{i=1}^{n} \e Z_{n,i}
        = \sum_{i=1}^{n} \p{\cap_{j=1}^{m} X_{n,j} \ne i}
        = \sum_{i=1}^{n} \left( 1-p_{n,i} \right)^{m}.
    \]

    For \(1 \le i \ne j \le n\), we have
    \begin{align*}
                        \E{Z_{n,i} Z_{n,j}} - \E{Z_{n,i}}\E{Z_{n,j}}
                &
        =
                (1-p_{n,i} - p_{n,j})^{m}
        - (1-p_{n,i})^{m} (1-p_{n,j})^{m}
        \\
        &
        =
        (1-p_{n,i})^{m}
        \left[
                \left(1-\frac{p_{n,j}}{1-p_{n,i}}\right)^{m}
        - (1-p_{n,j})^{m}
        \right]
        < 0.
    \end{align*}
    Therefore
    \begin{align*}
        \V{Z_n}
        &
        =
        \sum_{1 \le i,j \le n} \E{Z_{n,i} Z_{n,j}} - \E{Z_{n,i}}\E{Z_{n,j}}
        \\
        &
        =
        \sum_{1 \le i \ne j \le n}
        \left(
        \E{Z_{n,i} Z_{n,j}} - \E{Z_{n,i}}\E{Z_{n,j}}
        \right)
        +
        \sum_{1 \le i \le n}
        \left(
        \E{Z_{n,i}} - \E{Z_{n,i}}^{2}
        \right)
        \le \e Z_{n}.
    \end{align*}
    Thus by Chebyshev's inequality, as in the second moment method
    (see e.g., \citet[chap.\ 4]{Alon2008}), we have
    \[
        \p{Z_n = 0}
        \le \p{|Z_n - \e Z_n| \ge \e Z_n}
        \le \frac{\V{Z_n}}{(\e Z_n)^2}
        \le \frac{1}{\e Z_n}
        =
        \frac{1}{\sum_{i=1}^n (1-p_{n,i})^{m_n}}
        .
        \qedhere
    \]
\end{proof}

\subsection{Complete $r$-ary fringe subtrees}

\label{sec:fringe:app:r:ary}

A tree \(T\) is called possible if \(\pi(T) > 0\).
Let \(r > 0\) be a fixed integer and \(h_n\) be a sequence of positive integers.
A simple application of Theorem \ref{intro:thm:fringe:count} is to find sufficient conditions
such that whp every (or not every) possible complete \(r\)-ary tree appears in \(\gwn\)
as fringe subtrees.

\newcommand{\rAryMaxHeight}{H_{n,r}}
\newcommand{\rArySet}{\cA_{h_n, r}}
\newcommand{\oneAryMaxHeight}{H_{n,1}}
\newcommand{\oneArySet}{\cA_{h_n, 1}}

Let \(h_n \to \infty\) be a sequence of positive integers.
Let \(\rArySet\) be the set of all \emph{possible} complete \(r\)-ary trees of height at most \(h_n\).
Let
\[
    \rAryMaxHeight
    = \max
    \left\{
        h: \gwn \text{ contains all trees in \(\rArySet\) as fringe subtrees}
    \right\}
    .
\]

\begin{lemma}
    \label{family:lem:r:ary}
    Assume Condition \ref{cond:offspring} and \(p_r > 0\) for some \(r \ge 2\).
    Let
    \[
        \alpha_r = \log_r
    \left(
        \log \frac{1}{p_0} + \frac{1}{r-1} \log \frac{1}{p_r}
        \right)
    .
    \]
    Let \(\omega_n \to \infty\) be an arbitrary sequence.
    \begin{enumerate}[label=(\roman*)]
        \item If \(h_n \le \log_r(\log n - \omega_n) - \alpha_r\), then whp \(\gwn\) contains all trees
            in \(\rArySet\) as fringe subtrees.
        \item If \(h_n \ge \log_r(\log n + \omega_n) - \alpha_r\), then whp \(\gwn\) does not contain
            all trees in \(\rArySet\) as fringe subtrees.
    \end{enumerate}
    Also, \[\rAryMaxHeight - \log_{r} \log n \inprob -\alpha_r.\]
\end{lemma}

\begin{proof}
    Let \(\treerary_{h_n}\) denote the complete \(r\)-ary tree of height \(h_n\). Note that if
    \(\treerary_{h_n}\) appears in
    \(\gwn\) as a fringe subtree, then every tree in \(\rArySet\) also appears in \(\gwn\) as a fringe
    subtree.
    The tree \(\treerary_{n}\) has \(\ell_n:=r^{h_n}\) leaves and
    \(v_n:= (r^{h_n}-1)/(r-1)=(\ell_n-1)/(r-1)\) internal vertices, which all have degree \(r\).
    Thus we have
    \[
        \pi(\treerary_{h_n}) := \p{\gw = \treerary_{h_n}} = p_{r}^{v_n} p_{0}^{\ell_n}
        .
    \]
    If \(h_n \le \log_r(\log n - \omega_n) - \alpha_r\), then
    \[
        \ell_n = r^{h_n} \le \frac{\log n - \omega_n}{r^{\alpha_r}}.
    \]
    Therefore
    \begin{align*}
        \log \frac{1}{\pi(\treerary_{h_n})}
        & =
        v_n \log \frac{1}{p_r} + \ell_n \log \frac{1}{p_0}
        \\
        & =
        \frac{\ell_n-1}{r-1} \log \frac{1}{p_r} + \ell_n \log \frac{1}{p_0}
        \\
        & =
        \ell_{n} \left( \frac{1}{r-1} \log \frac{1}{p_r} + \log \frac{1}{p_0} \right) +O(1)
        \\
        &
        \le
        \frac{\log n - \omega_n}{r^{\alpha_r}} r^{\alpha_r} + O(1)
        \\
        &
        = \log n - \omega_n + O(1).
    \end{align*}
    Thus \(\log (n \pi(\treerary_{h_n})) \ge \omega_n + O(1) \to \infty\), which implies that \(n
    \pi(\treerary_{h_n}) \to \infty\).
    It follows from Theorem \ref{intro:thm:fringe:count} that \(\treecount_{\treerary_{h_n}}(\gwn) \inprob \infty\).
    Thus (i) is proved.

    Similar computations show that with the assumptions of (ii), we have \(n
    \pi(\treerary_{h_n}) \to 0\), which implies that \(\treecount_{\treerary_{h_n}}(\gwn) \inprob 0\) by Theorem \ref{intro:thm:fringe:count}.
    The last statement of the lemma follows directly from (i) and (ii).
\end{proof}

We have a similar result for the set of \(1\)-ary trees (chains) of height at most \(h\).
The proof is virtually identical to the previous lemma and we leave it to the reader.
\begin{lemma}
    \label{family:lem:fringe:chain}
    Assume Condition \ref{cond:offspring} and \(p_1 > 0\).
    Let \(\omega_n \to \infty\) be an arbitrary sequence.
    We have:
    \begin{enumerate}[label=(\roman*)]
        \item If \(h_n \le (\log n - \omega_n)/\log \frac{1}{p_1}\), then whp \(\gwn\) contains
            all trees in \(\oneArySet\) as fringe subtrees.
        \item If \(h_n \ge (\log n + \omega_n)/\log \frac{1}{p_1}\), then whp \(\gwn\) does not contain
            all trees in \(\oneArySet\) as fringe subtrees.
    \end{enumerate}
    Therefore \[\frac{\oneAryMaxHeight}{\log_{{1}/{p_1}}(n)} \inprob 1.\]
\end{lemma}

\subsubsection{Binary trees}
\label{family:sec:bin:trees}

    Consider \(\gwn\) with a binomial \( (2,1/2)\) offspring distribution, i.e., \(p_0=p_2=1/4\) and \(p_1=1/2\).
    Let \(\gwbin\) be \(\gwn\) with each degree-one node labeled of having a left or a right child
    uniformly and independently at random.
    Then \(\gwbin\) is a tree in which nodes have a left position and a right position where child
    nodes can attach, and each position can be occupied by at most one child.
    We call such a tree a \emph{binary tree}.
    Figure \ref{family:fig:binary:tree} gives an example of \(\gwbin\) for
    \(n=7\).
    \begin{figure}[ht!]
    \centering
        \begin{tikzpicture}
        \node[anchor=south west,inner sep=0] at (0,0) {\includegraphics{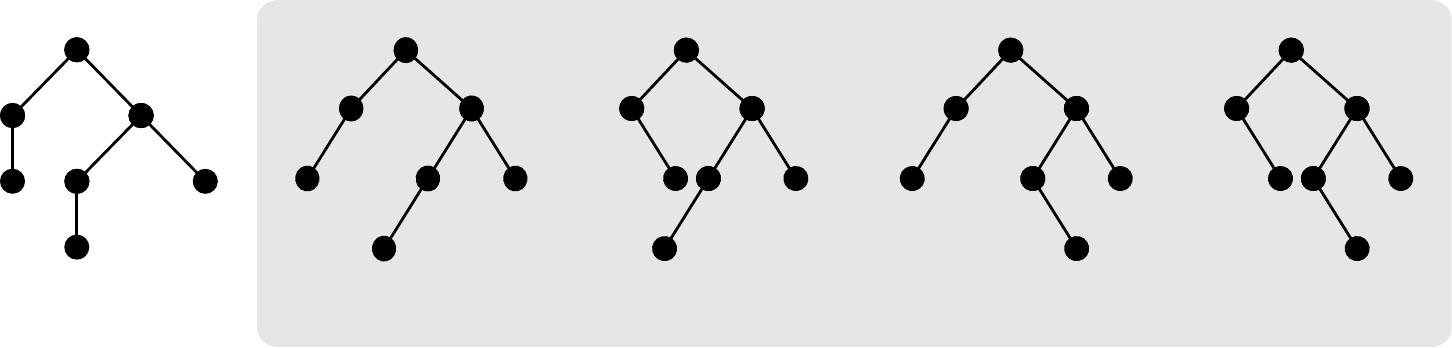}};
        \node at (1.2,0.4) {\small $T$};
        \node at (8.5,0.4) {\small Possible outcomes of $\gwbi_7$ given that \(\gw_{7}=T\)};
        \end{tikzpicture}
        \caption{Example of \(\gwbin\).}
        \label{family:fig:binary:tree}
    \end{figure}

    Let \(\Tnbin\) be a binary tree of size \(n\) that has \(n_0,n_1,n_2\) nodes of degree
    \(0,1,2\) respectively, where \(n_0+n_1+n_2=n\). Let \(T_n\) be \(\Tnbin\) with the difference between left and right
    children being forgotten.
    We have \(\p{\gwbin = \Tnbin | \gwn = T_n} = 1/2^{n_1}\), since there are in total
    \(2^{n_1}\) ways to label the \(n_1\) degree-one nodes of \(T_n\).
    Therefore
    \begin{align*}
        \p{\gwbin = \Tnbin}
        &
        =
        \p{\gwbin = \Tnbin | \gwn = T_n}
        \p{\gwn = T_n}
        \\
        &
        =
        \p{\gwbin = \Tnbin | \gwn = T_n}
        \frac{\p{\gw = T_n}}{\p{|\gw|=n}}
        \\
        &
        = \frac{1}{2^{n_1}}\frac{1}{4^{n_0 n_2}2^{n_1}} \frac{1}{\p{|\gw|=n}}
        = \frac{1}{4^{n}\p{|\gw|=n}}
        .
    \end{align*}
    In other words, as is well-known from the connection between simply generated trees and
    Galton-Watson trees of size \(n\), see, e.g., \citet[pp.\ 132]{janson2012simply},
    \(\gwbin\) is uniformly distributed among all binary trees of size \(n\).

    Thus our analysis of maximum \(r\)-ary fringe subtree in \(\gwn\) can be easily adapted to
    uniform random binary trees.  For example, an argument similar to Lemma
    \ref{family:lem:fringe:chain} shows that the maximum one-ary fringe subtree (chain) in
    \(\gwbin\) that consists of only left children has length about \(\log_{4} n\).
\subsection{All possible fringe subtrees}

\label{sec:fringe:app:all:possible}

Recall that \(\tsetplus{k}\) denotes the set of all possible trees of size at most \(k\), i.e.,
\[
    \tsetplus{k} := \{T \in \tset : |T| \le k, \p{\gw = T} > 0\}.
\]
Also recall that
\[
    \pSetMaxSize :=
    \max
    \{
        k: \tsetplus{k} \subseteq \cup_{v \in \gwn} \{\gw_{n,v}\}
    \}
    .
    \]
We would like to study the growth of \(\pSetMaxSize\) with \(n\).

Let \(\Ge(p)\) denote Geometric \(p\) distribution, i.e., \(\p{\Ge(p) = i} = p (1-p)^i\) for \(i \in \N_0\).
Let \(\Be(p)\) denote Bernoulli \(p\) distribution and let \(\Bi(d, p)\) denote Binomial
\( (d, p)\) distribution.
Recall that \(\Po(\lambda)\) is a Poisson\((\lambda)\) random variable.
Table \ref{table:distributions} shows five types of well-known conditional Galton-Watson trees.
See \citet[sec.\ 10]{janson2012simply} for more examples.

\begin{table}[h!]
  \centering
  \begin{tabular}{ | c | >{$}l<{$} >{$}l<{$} |}
  \hline
  Name & \multicolumn{2}{c|}{Definition} \\
  \hline
    Plane trees   & \xi \eql \Ge(1/2) & p_i=1/2^{i+1} \quad (i \ge 0) \\
    Full binary trees & \xi \eql 2 \Be(1/2) &  p_0=p_2 = 1/2 \\
    Motzkin trees & \xi \text{ uniform in } \{0,1,2\} & p_0=p_1=p_2 = 1/3 \\
    \(d\)-ary trees & \xi \eql \Bi(d, 1/d)  & p_i = \binom{d}{i}\left(1-\frac 1 d\right)^{i} \left(\frac 1 d\right)^{d-i} \quad (0 \le i \le d)  \\
    Labeled trees & \xi \eql \Po(1)  & p_i=e^{-1}/{i!} \\
  \hline
\end{tabular}
\caption{Some well-known conditional Galton-Watson trees.}
\label{table:distributions}
\end{table}

We can assume that \(\p{|\gw| = k_n} > 0\) for all \(n \in \N\).
Otherwise, let \(k_n' := \max\{i \le k_n: \p{|\gw| = i} > 0\}\).
It is not difficult to show that \(k_n - k_n' \le h = O(1)\) for \(k_n\) large.
(See \cite[lem.\ 12.3]{janson2012simply} for details.)
Thus this assumption does not change results in this
subsection.

\citeauthor{janson2012simply} showed that \(\gwrand \inlaw \gw\)
\cite[thm.\ 7.12]{janson2012simply}.
In other words, fringe subtrees on average behave like unconditional
Galton-Watson trees.
Let \(\tleast\) be a tree \(T \in \tsetplus{k}\) that minimizes \(\p{\gw = T}\).
Then \(\tleast\) also is the least likely tree to appear in \(\gwn\) as
fringe subtree among all trees in \(\tsetplus{k}\) when \(n\) is large. So intuitively if whp \(\tleast\) appears in
\(\gwn\), then every tree in \(\tsetplus{k}\) should also appears whp. And if whp \(\tleast\) does not appear,
then whp there is at least one tree in \(\tsetplus{k}\) that is missing.
Therefore, the problem can be reduced to finding
\[
    \pleastk := \p{\gw = \tleast} = \min_{T \in \tsetplus{k}} \p{\gw = T}.
\]

\begin{lemma}
    \label{faimly:lem:all:appear:condition}
    Assume Condition \ref{cond:offspring}.
    If \(k_n \to \infty\) and
    \( n \pleast_{k_n}/k_n \to \infty \),
    then  whp every tree in
    \(\tsetpluskn\) appears in \(\gwn\).
\end{lemma}

\begin{proof}    Let \(k = k_n\).
    Recall that \(\pmax := \max_{i \ge 0} p_i\) and \(\pmax < 1\).
    Therefore \(\pleastk \le \pmax^{k}\).
    Thus we can assume that \(k \le 2 \log(n)/\log(\pmax^{-1})\) when \(k\) is large.
    Otherwise we have \(n \pleastk \le n^{-1} \to 0\), which contradicts the assumption.

    Thus by Theorem \ref{intro:thm:fringe:set:count},
    \(\fringeSk \ge y_n := \ceil{\frac{1}{2} n \p{\gw = |k|}}\) whp.
    Let \(A_n\) be the event that \(\gwn\) contains all possible trees of size \(k\) as fringe
    subtrees.
    Let \(B_n(i)\) be the event that \(\fringeSk = i\) for given
    \(i \ge y_n\).
    Thus
    \begin{align*}
        \p{A_n}
        &
        =
        \p{A_n \cap [\fringeSk < y_n]} +
        \sum_{i \ge y_n} \p{A_n | B_n(i)} \p{B_{n}(i)}
        \\
        &
        \ge
        \p{A_n | B_n(y_n)} \sum_{i \ge y_n} \p{B_{n}(i)},
        \\
        &
        =
        \p{A_n | B_n(y_n)}(1-o(1)),
    \end{align*}
    where the inequality comes from the obvious fact that \(\p{A_n|B_n(a)} \ge \p{A_n|B_n(b)}\)
    given that \(a \ge b\).
    So it suffices to prove that \(\p{A_n^c|B_n(y_n)} \to 0\).

\newcommand{\gwstar}{\gwraw^{*}}

    We can show that this is equivalent to a coupon collector problem.
    Let \( \gwn|B_n(y_n) \) be \(\gwn\) restricted to the event \(B_n(y_n)\).
        Let \(\gwstar\) be a random tree distributed as \(\gwn|B_{n}(y_n)\).
    Replace each of its \(y_n\) subtrees with an independent copy of \(\gw_k\).
    The result is still a random tree distributed as \( \gwn|B_n(y_n)\). (We leave the verification
    to the readers.)

    So \(\p{A_n|B_n(y_n)}\) equals the probability of that \(y_n\) independent copies of
    \(\gw_k\) contain every tree in \(\tsetplus{k}\).
    It follows from Lemma \ref{family:lem:coupon} (the coupon collector) that
    \begin{align*}
        \p{A_n^c|B_n(y_n)}
        &
        \le
        \sum_{T \in \tsetplus{k}} \left( 1-\p{\gw_k = T} \right)^{y_n}
        \le
        \sum_{T \in \tsetplus{k}}
        \exp\left\{
            - y_n \p{\gw_k = T}
        \right\}
        \\
        &
        \le
        \sum_{T \in \tsetplus{k}}
        \exp\left\{
            - \frac{1}{2} n \p{|\gw| = k}
            \frac{\p{\gw = T}}{\p{|\gw| = k}}
        \right\}
        \\
        &
        \le
        O(|\tsetplus{k}|)
        \exp\left\{
            - n \pleastk
        \right\}
        .
    \end{align*}
    It is well-known that the number of plane trees of size exactly \(k\) is \(4^{k-1}/\sqrt{\pi
        k^{3}}(1 + o(1))\)
    See, e.g., \citet[pp.\ 406]{Flajolet2009}.
    It follows that there exists a constant \(C\) such that
    \(
    |\tsetplus{k}| \le \sum_{j=1}^{k} C {4^{j-1}} \le C 4^{k}
    \)
    for all \(k \in \N\).
    Thus for large enough \(k\), the last expression above is at most
    \begin{align*}
        C4^{k}
        \exp\left\{
            - n \pleastk
        \right\}
        =
        C
        \exp\left\{
            k \log(4)
            - {n} \pleastk
        \right\}
        \to 0
        .
    \end{align*}
    Therefore we have \(\p{A_n^c|B_n(y_n)} \to 0\).
\end{proof}

\begin{theorem}
    \label{family:thm:fringe:all:appear}
    Assume Condition \ref{cond:offspring}.
    Assume that as \(k \to \infty\),
    \[
        {\log (1/ {\pleastk} )} \sim {\gamma k^{\alpha} (\log k)^{\beta}},
    \]
    where \(\alpha \ge 1\), \(\beta \ge 0\), \(\gamma>0\) are constants.
    Let \(k_n \to \infty\) be a sequence of positive integers.
    Let \(m = \log n\).
    Then for all constants \(\delta > 0\), we have:
    \begin{enumerate}[label=(\roman*)]
        \item If
            \(k_n \le
            (1-\delta) [m / \gamma (\log m^{1/\alpha})^{\beta}
                            ]^{1/\alpha}
                ,
            \)
            then whp \(\gwn\) contains
            all trees in \(\tsetpluskn\) as fringe subtrees.
        \item If
            \(k_n \ge
            (1+\delta) [m / \gamma (\log m^{1/\alpha})^{\beta}
                            ]^{1/\alpha}
                ,
            \)
            then whp \(\gwn\) does not contain
            all trees in \(\tsetpluskn\) as fringe subtrees.
    \end{enumerate}
    As a result,
    \[
        \frac{K_n}{\left( {\log n}/{(\log \log n)^{\beta}} \right)^{1/\alpha}}
        \inprob 
        \left(
            \frac{\alpha^{\beta}}{\gamma}  
        \right)^{1/\alpha}.
    \]
\end{theorem}

The behavior of \(\pleastk\) varies for different offspring distributions.  But as mentioned in the
introduction, the types of trees that we are interested in all have \(\pleastk\) that
are covered by Theorem \ref{family:thm:fringe:all:appear}.
\begin{proof}
    Let \(k = k_n\).
    Part (i) assumes that
    \[
        k \le (1-\delta) \left[\frac m { \gamma (\log m^{1/\alpha})^{\beta}} \right]^{1/\alpha}.
        \]
    Taking a logarithm, we have
    \begin{align*}
        \log k
        &
        \le
        \log(1-\delta)
        +
        \frac{1}{\alpha}
        \log
        \frac{m}{ \gamma (\log m^{1/\alpha})^{\beta}}
                = (1+o(1))
        \log m^{1/\alpha}.
    \end{align*}
    Thus
    \begin{align*}
        \log \frac{1}{\pleastk}
        &
        \sim \gamma k^{\alpha} (\log k)^{\beta}
        \le (\gamma +o(1))
        \frac{(1-\delta)^{\alpha} m}{\gamma (\log m^{1/\alpha})^{\beta}}
        (\log m^{1/\alpha})^{\beta}
                        \sim  (1-\delta)^{\alpha} m.
    \end{align*}
    Therefore, recalling \(m = \log n\),
    \begin{align*}
        \log n\pleastk
        =
        m - \log\frac{1}{\pleastk}
        \ge
        m - (1+o(1)) (1-\delta)^{\alpha}m
        = \Omega(m).
    \end{align*}
    It follows that
    \begin{align*}
        \log k
                        -
        \log n \pleastk
        =
        O\left( \log m  \right)
        - \Omega(m) \to -\infty.
    \end{align*}
    Thus \(n \pleastk/ k \to \infty\) and it follows from Lemma
    \ref{faimly:lem:all:appear:condition} that whp \(\gwn\) contains every tree in \(\tsetplus{k}\) as a fringe
    subtree.

    Similar computations show that if
    \(k \ge (1+\delta) [m / \gamma (\log m^{1/\alpha})^{\beta} ]^{1/\alpha}\)
    then \(n \pleastk \to 0\).
    It follows from Theorem \ref{intro:thm:fringe:count} that \(\treecount_{\tleast}(\gwn) \inprob 0\).
    Thus whp \(\gwn\) does not contain every tree in \(\tsetplus{k}\) as a fringe subtree.
\end{proof}

The rest of this section is organized as follows.
In the next subsection, we give a general method of finding \(\pleastk\).
Then we divide offspring distributions in two categories and show that Theorem
\ref{family:thm:fringe:all:appear} is applicable to all the Galton-Watson trees listed in Table
\ref{table:distributions}.

\subsubsection{Computing \(\pleastk\)}

Let \(\cI_{k} := \{j: 1 \le j \le k, p_j > 0\}\).
Let \(L_1 := p_0\) and for \(k \ge 2\) let
\[
    L_k := \min \left\{ p_0 \left( \frac{p_i}{p_0} \right)^{1/i} : i \in \cI_{k-1} \right\}.
\]
Since \(L_{k}\) is non-increasing, \(L := \lim_{k \to \infty} L_k\) exists.
Equivalently, we have
\[
    L := \inf \left\{ p_0 \left( \frac{p_i}{p_0} \right)^{1/i} : i \in \N, p_i > 0 \right\}.
\]
\begin{theorem}
    \label{family:thm:pleast:and:L}
    Assume Condition \ref{cond:offspring}.
    We have
    \(
        \pleastToOneByk \to L
    \)
    as \(k \to \infty\),
    where the limit is taken along the subsequence \(k\) with \(\p{|\gw| = k} > 0\).
    As a result, we have \(L < 1\).
\end{theorem}

In fact, we have a stronger result for the upper bound of \(\pleastToOneByk\).
\begin{lemma}
    \label{family:lem:many:unlikely:tree}
    Assume Condition \ref{cond:offspring}.
    For all fixed \(i\) with \(p_i > 0\), there exist constants \(C_i > 1\)
    and \(C_{i}', C_{i}'', k(i) > 0\)
    such that
    for all \(k \ge k(i)\) with \(\p{|\gw|=k}>0\),
    there are at least \(k^{-C_i'} C_i^{k}\) trees \(T\) of size \(k\) with
    \[
        0 < \p{\gw = T} \le C_i''
        \left[
            p_0 \left(\frac {p_i}{p_0}\right)^{1/i}
        \right]^{k}
        .
    \]
\end{lemma}

\begin{proof}
    We assume that \(h\) (the span of \({\xi}\)) is \(1\).  The proof for \(h > 1\) is similar.
    Let \(j\) be the smallest positive integer such that \(j\) is coprime with \(i\) and \(p_j > 0\).
    (Such a \(j\) exists except for the trivial case that \(p_0+p_i=1\).)
    Let \(x = (k-1) \bmod i\).
    By the Chinese reminder theorem, there exists a smallest non-negative integer \(y\) such that
    \[
        \begin{cases}
            y \equiv x \pmod i, \\
            y \equiv 0 \pmod j.
        \end{cases}
    \]
    Note that \(y\) depends only on \(i\).
    Therefore, if \(k \ge k(i):=y+1\), we can choose
    \[
        n_{0} = k - n_{i} - n_{j}
        ,
        \qquad
        n_{j} = \frac{y}{j},
        \qquad
        n_{i} = \frac{k-1-y}{i},
    \]
    such that \(n_{0}, n_{i}, n_{j}\) are all non-negative integers with
    \[
        n_0 + n_{i} + n_{j} = k
        ,
        \qquad
        \text{and}
        \qquad
        i n_{i} + j n_{j} = k-1
        .
    \]
    Let \(\tsetk(n_0, n_i, n_j)\) be the set of plane trees of size \(k\) that has \(n_{0}\),
    \(n_i\) and \(n_j\) nodes with degree \(0\), \(i\) and \(j\) respectively.
    It is well-known that when the above two conditions hold, we have
    \begin{align*}
        |\tsetk(n_0, n_i, n_j)|
        :=
        \frac{1}{k} \binom{k}{n_0,n_{i}, n_{j}}
        =
        \frac{1}{k} \frac{k!}{n_0 ! n_{i} ! n_{j} !}.
    \end{align*}
    (See \cite[pp.\ 194]{Flajolet2009}.)
    Since \(i\) is a constant and \(y\) only depends on \(i\),
    there exists a constant \(C_i^{*}\) such that
    \[
        \left|n_{0} - k(1-1/i) \right| \le  C_i^{*}
        ,
        \qquad
        |n_{i} - k/i| \le C_i^{*}
        ,
        \qquad
        n_{j} \le C_i^{*}
        .
    \]
    Using these inequalities and Stirling's approximation \cite[pp.\ 407]{Flajolet2009}, it is easy to verify that
    there exists a constant \(C_{i}'>0\) such that
    \[
        |\tsetk(n_0, n_i, n_j)|
        \ge
        k^{-C_i'}
        \left( \left( \frac{1}{i} \right)^{1/i} \left( 1-\frac{1}{i} \right)^{1-1/i}
        \right)^{-k}
        := k^{-C_i'} C_i^{k}
        .
    \]
    And for every \(T \in \cS(n_0, n_i, n_j)\), we have
    \[
        \p{\gw = T}
        \le
        p_{i}^{n_{i}} p_{0}^{n_0}
        \le
        p_{i}^{-C_i^{*}} p_{0}^{- C_i^{*}}
        p_{i}^{k/i} p_{0}^{k(1-1/i)}
        :=
        C_{i}''
        \left[ p_0 \left( \frac{p_{i}}{p_0} \right)^{1/{i}}  \right]^{k}
        .
        \qedhere
        \]
\end{proof}

\begin{proof}[Proof of Theorem \ref{family:thm:pleast:and:L}]
    Let \(T\) be a tree with \(|T|= k\) and \(\p{\gw = T} > 0\), i.e., \(T \in \tsetplus{k}\).
    Let \(n_i\) be the number of nodes of degree \(i\) in \(T\).
    Note that if \(i > 0\) and \(i \notin \cI_{k-1}\), then \(n_i = 0\).
    Since by \eqref{pre:eq:deg:seq} the sum of the degrees in a preorder degree sequence equals the size of the tree
    minus one, we have
    \begin{equation*}
        n_0 + n_1 + \ldots n_{k-1} = k,
        \qquad
        \text{ and }
        \qquad
        n_1 + 2 n_2 \ldots + (k-1) n_{k-1} = k-1.
            \end{equation*}
    Using the convention that \(0^{0} = 1\), we have for \(k \ge 2\)
    \begin{align*}
        \p{\gw = T}
        & =
        p_{0}^{n_0} p_{1}^{n_{1}} \cdots p_{k-1}^{n_{k-1}}
        \\
        & =
        p_{0}^{n_0 + n_1 + \ldots + n_{k-1}}
        \left( \frac{p_1}{p_0} \right)^{n_1}
        \left( \frac{p_2}{p_0} \right)^{n_2}
        \cdots
        +
        \left( \frac{p_{k-1}}{p_0} \right)^{n_{k-1}}
        \\
        & =
        p_{0}^{k} \prod_{i \in \cI_{k-1}}
        \left[
            \left( \frac{p_i}{p_0} \right)^{\frac{1}{i}}
        \right]^{i n_{i}}
        \ge
        p_{0}^{k}
        \left[
            \min_{i \in \cI_{k-1}}
            \left( \frac{p_i}{p_0} \right)^{\frac{1}{i}}
            \right]^{\sum_{i=1}^{k-1} i n_{i}}
        \\
        &
        = p_0 L_{k}^{k-1}
        \ge p_0 L^{k-1}.  \numberthis \label{eq:fringe:pleast:lower}
    \end{align*}
    As a result
    \(
        \liminf_{k \to \infty} \pleastToOneByk \ge L.
    \)

\newcommand{\iup}{\alpha}
\newcommand{\idown}{\beta}

    To show the other way, let \(\varepsilon > 0\) be a constant, and let
    \(
    \iup = \min\{i:L_{i+1} \le L + \varepsilon\}.
    \)
    Therefore
    \( 0 < p_0 (p_\iup/p_0)^{1/\iup} \le L + \varepsilon  \).
    By Lemma \ref{family:lem:many:unlikely:tree}, there is at least one tree \(T\) of size
    \(k\) such that
    \begin{align*}
        \pleastk
        &
        \le
        \p{\gw = T}
        \le
        C_{\alpha}
        \left[ p_0 \left( \frac{p_{\iup}}{p_0} \right)^{\frac 1 {\iup}}  \right]^{k}
        \le
        C_{\alpha} (L+\varepsilon)^{k},
    \end{align*}
    where \(C_{\alpha} > 0\) is constant.
    Thus
    \(
        \limsup_{k \to \infty} \pleastToOneByk \le L+\varepsilon.
    \)
    Since \(\varepsilon\) is arbitrary,
    we have
    \(\limsup_{k \to \infty} \pleastToOneByk \le L\).

    Recall that \(\pmax := \max_{i \ge 0} p_i < 1\). For all trees \(T\) with size \(k\), we have \(\p{\gw = T} \le
    \pmax^{k},\) i.e., \(\pleastToOneByk \le \pmax\). Thus \(L = \lim_{k \to \infty} \pleastToOneByk \le
    \pmax < 1\).
\end{proof}

\subsubsection{When \(L > 0\)}

If \(L > 0\), then by Theorem \ref{family:thm:pleast:and:L},
\(\log (1/\pleastToOneByk) \sim  \log(1/L) k = \log(1/L)k(\log k)^0\).
Thus we can apply Theorem \ref{family:thm:fringe:all:appear} with \(\gamma = \log(1/L)\),
\(\alpha = 1\) and \(\beta=0\) to get
\[\frac{\pSetMaxSize}{\log(n)} \inprob \frac{1}{\log\left(1/L \right)}.\]
The following Lemma computes \(L\) for some well-known Galton-Watson trees.
See \citet[sec.\ 10]{janson2012simply} for more about these trees.

\begin{lemma}
    \label{family:lem:L:of:4}
        (i) Full binary tree: If \(\xi \eql 2 \Be(1/2)\), then \(L = 1/2\).
        (ii) Motzkin tree: If \(p_0 = p_1 = p_2 = 1/3\), then \(L = 1/3\).
        (iii) \(d\)-ary tree: If \(\xi \eql \Bi(d,1/d)\) for \(d \ge 2\), then \(L = (d-1)^{d-1}/d^d\).
        (iv) Plane tree: If \(\xi \eql \Ge(1/2)\), then \(L = 1/4\).
\end{lemma}

\begin{proof}

    (i): If \(\xi \sim 2 \Be(1/2)\), then \(p_0 = 1/2\), \(p_2=1/2\) and \(p_i = 0\) for \(i \notin
    \{0,2\}\).
    Thus for \(k \ge 3\), we have
    \[
        L_{k} = \min_{i:i< k, p_i > 0} p_0 \left( \frac{p_i}{p_0} \right)^{1/i}
        = p_0 \left( \frac{p_2}{p_0} \right)^{1/2} = 1/2.
        \]
        Therefore \(L = \lim_{k \to \infty} L_{k} = 1/2\).
    (ii) and (iii) follow from similar simple calculations.

    (iv): For all \(i \ge 1\), we have
    \[
        p_{0} \left( \frac{p_{i}}{p_{0}} \right)^{1/i}
        = \frac{1}{2} \left( \frac{1}{2^{i}} \right)^{1/i} = 1/4.
    \]
    Therefore \(L_{k} = 1/4\) for all \(k \ge 1\), and \(L = 1/4\).
\end{proof}


Define
\[
    \kappa =
    \begin{cases}
        \min \{i \in \N: L_{i+1} = L\} & \text{if \(L_j = L\) for some \(j\)}, \\
        \infty & \text{otherwise}.
    \end{cases}
\]
If \(\kappa < \infty\), then we call the Galton-Watson tree \emph{well-behaved}.
Examples of such trees include those for which \(\xi\) is bounded,
and those for which \(\xi\) has a polynomial or sub-exponential tail.
The case \(\xi \eql \Ge(1/2)\) is also well-behaved.
Thus the four types of Galton-Watson trees in Lemma \ref{family:lem:L:of:4} are
well-behaved. The following theorem gives better thresholds than Theorem
\ref{family:thm:fringe:all:appear}.

\begin{theorem}
    \label{family:thm:well:behaved}
    Assume Condition \ref{cond:offspring} and let the Galton-Watson tree be well-behaved.
    Then for all constants \(\delta > 0\), we have:
    \begin{enumerate}[label=(\roman*)]
        \item If \(k_n \le (\log n - (1+\delta)\log \log n)/\log \frac{1}{L}\), then whp \(\gwn\) contains
            all trees in \(\tsetpluskn\) as fringe subtrees.
        \item If \(k_n \ge (\log n - (1-\delta)\log \log n)/\log \frac{1}{L}\),
            then whp \(\gwn\) does not contain
            all trees in \(\tsetpluskn\) as fringe subtrees.
    \end{enumerate}
    Thus as \(n \to \infty\), we have
    \[
        \frac{K_n \log(1/L)- \log n}{\log \log n} \inprob -1.
    \]
\end{theorem}

The main idea is that when \(\ileast < \infty\), there are exponentially many trees of size
\(k\) that have small probability to appear as fringe subtrees in \(\gwn\).
Then we can use Lemma \ref{family:lem:coupon} (the coupon collector) to find the sufficient condition for one of them to not to
appear whp.

\begin{proof}    (i): Write \(m = \log n\)
    and \(k = k_n\).
    Using \eqref{eq:fringe:pleast:lower},
    it is easy to verify that in this case \(n \pleastk/k \to \infty\).
    Thus (i) follows from  Lemma \ref{faimly:lem:all:appear:condition}.

    (ii): The proof is similar to the one of Lemma \ref{faimly:lem:all:appear:condition}.
    As in that proof, we can assume that \(k = O(\log n)\).
    Thus by Theorem \ref{intro:thm:fringe:set:count}, whp \(\fringeSk \le y_n:=
    \floor{\frac{3}{2} n \p{|\gw|=k}}\).
        Let \(A_n\) be the event that \(\gwn\) contains all possible trees of size \(k\) as fringe
    subtrees.
    Let \(B_n(i)\) be the event that \(\fringeSk = i\) for some \(i \le y_n\).
    Then
    \begin{align*}
        \p{A_n}
        &
        \le
        \p{\fringeSk > y_n} +
        \sum_{i \le y_n} \p{A_n | B_n(i)} \p{B_{n}(i)}
        \\
        &
        \le
        o(1)
        +
        \p{A_n | B_n(y_n)} \sum_{i \le y_n} \p{B_{n}(i)},
        \\
        &
        \le
        o(1) +  \p{A_n | B_n(y_n)}.
    \end{align*}
    Thus it suffices to prove that \(\p{A_n|B_n(y_n)} \to 0\).

    Using the same coupling as in the proof of Lemma \ref{faimly:lem:all:appear:condition}, we have
    \(\p{A_n|B_n(y_n)}\) equals the probability that \(y_n\) independent copies of \(\gw_k\) do
    not contains all trees in \(\tsetplus{k}\). It follows from Lemma \ref{family:lem:coupon} (the coupon collector)
    that \(\p{A_n|B_n(y_n)} \to 0\) if \(\sum_{T \in \tsetplus{k}} (1-\p{\gw_k=T})^{y_n}\) goes to infinity.

    By definition of \(\ileast\), we have \(L = p_0 (p_{\ileast}/p_0)^{1/\ileast}\).
    It follows from Lemma \ref{family:lem:many:unlikely:tree} that there
    exists constants \(C_\ileast> 1\) and \(C_{\ileast}', C_{\ileast}'' >0\)
    such that there
    are at least \(k^{-C_{\ileast}'}
    C_{\ileast}^k\) trees \(T\) in \(\tsetplus{k}\) with
    \begin{align*}
        \p{\gw_k = T} = \frac{\p{\gw = T}}{\p{|\gw|=k}}
        \le
        C_{\ileast}''
        \frac{(p_0 (p_{\ileast}/p_0)^{1/\ileast})^{k}}{\p{|\gw|=k}}
        =
        \frac{C_{\ileast}''L^k}{\p{|\gw|=k}}
        .
    \end{align*}
    Therefore
    \begin{align*}
    \sum_{T \in \tsetplus{k}} (1-\p{\gw_k=T})^{y_n}
    &
    \ge
    k^{-C_{\ileast}'} C_{\ileast}^{k} \left( 1-\frac{C_{\ileast}''L^k}{\p{|\gw|=k}} \right)^{y_n}.
    \end{align*}
    Since \(L<1\) and \(\p{|\gw|=k}=\Theta(k^{-3/2})\), we have
    \(L^k/\p{|\gw|=k} = o(1)\).
    Thus for \(k\) large enough, the logarithm of the above is
    \begin{align*}
        & k \log(C_{\kappa})
        - C_{\kappa}' \log(k)
        + y_n \log\left( 1-\frac{C_{\kappa}'' L^{k}}{\p{|\gw|=k}} \right)
        \\
        &
        \ge
        \frac{1}{2}k \log(C_{\kappa})
        - y_n \frac{C_{\kappa}'' L^{k}}{\p{|\gw|=k}}
        \\
        &
        \ge
        \frac{1}{2}k \log(C_{\kappa})
        - \frac{3}{2} n C_{\kappa}'' L^{k}
        =
        \frac{1}{2}k \log(C_{\kappa})
        - O(nL^{k}).
    \end{align*}
    By our assumptions, \(k = \Omega(\log n)\) and \(L^{k} \le (\log n)^{1-\delta}/n\).
    Since \(C_{\kappa}>0\), we have
    \[
        \frac{1}{2} k \log(C_{\kappa})-O(n L^{k})
        \ge \Omega(\log n) - \bigO{n \frac{(\log n)^{1-\delta}}{n}}
        \to \infty,
    \]
    which implies \(\p{A_n|B_n(y_n)} \to 0\).
\end{proof}

\begin{myRemark*}
    If \(L > 0\) and \(\ileast = \infty\), then Theorem \ref{family:thm:fringe:all:appear} shows that
    \(K_n/\log(n) \inprob 1/\log(1/L)\).  But the second order term of \(K_n\) is sensitive to small
    modifications of the offspring distribution, which makes it slightly more challenging to analyze
    the second order term.
\end{myRemark*}

\subsubsection{When \(L = 0\)}

It is clear that \(L=0\) if and only if \(\xi\) has infinite support and \(\lim\inf_{i \to \infty} p_0
(p_i/p_0)^{1/i} = 0\),
which implies \(\limsup_{i \to \infty} \log(1/p_{i})/i = \infty\), along the subsequence with \(p_i > 0\).
If in addition we have \(p_i > 0\) for all \(i \ge 0\) and
\(\log (1/p_i) \sim f(i)\) for some \(f:[0,\infty)\to[0,\infty)\) with \(f(i)/i \uparrow \infty\),
then we say that \(\xi\) has an \emph{\(f\)-super-exponential tail}.
We have the following threshold for
Galton-Watson trees with such a property.
\begin{theorem}
    \label{family:thm:super:exp}
    Assume Condition \ref{cond:offspring} and that \(\xi\) has an \(f\)-super-exponential tail.
    Let \(f^{-1}\) denote the inverse of \(f\).
    Then for all constants \(\delta > 0\), we have
    \begin{enumerate}
        \item If \(k_n \le f^{-1}( (1-\delta) \log n)+1\), then whp \(\gwn\) contains
            all trees in \(\tsetpluskn\) as fringe subtrees.
        \item If \(k_n \ge f^{-1}( (1+\delta) \log n)+1\),
            then whp \(\gwn\) does not contain
            all trees in \(\tsetpluskn\) as fringe subtrees.
    \end{enumerate}
    Therefore, \[\frac{K_n}{f^{-1}(\log n)} \inprob 1.\]
\end{theorem}

\begin{proof}
    (i): Let \(k = k_n\).
    Choose \(\varepsilon > 0\) such that \( (1-\delta)(1+\varepsilon) < (1-\delta/2)\).
    Since \(\log(1/p_i) \sim f(i)\), there exists an integer \(i(\varepsilon)\) such that
    for all \(i > i(\varepsilon)\),
    \[
        \log p_i \ge -(1+\varepsilon/2) f(i).
    \]
    Let \(w_i := \log p_0 (p_i/p_0)^{1/i}\).
    We have as \(k \to \infty\),
    \begin{align*}
        \min_{i(\varepsilon) < i < k}
        w_i
        &
        =
        \min_{i(\varepsilon) < i < k}
        \left( 1-\frac{1}{i} \right) \log(p_0) + \frac{\log p_i}{i}
        \\
        &
        \ge
        \log(p_0)
        -
        \max_{i(\varepsilon) < i < k}
        \frac{(1+\varepsilon/2) f(i)}{i}
        \\
        &
        = \log(p_0) - \frac{(1+\varepsilon/2) f(k-1)}{k-1}
        \to - \infty
        ,
    \end{align*}
    where we use that \(f(i)/i \uparrow \infty\).
    Since \(\min_{1 \le i \le i(\varepsilon)} w_i\) is a constant, we have for large \(k\),
    \begin{align*}
        \log L_k := \min_{1 \le i < k} w_i \ge
        \log(p_0)
        - \frac{(1+\varepsilon/2) f(k-1)}{k-1}.
    \end{align*}
    It follows from \eqref{eq:fringe:pleast:lower} that for \(k\) large enough,
    \begin{align*}
        \log \pleastk
        &
        \ge \log (p_0 L_k^{k-1})
        \\
        &
        \ge \log(p_0) + (k-1) \log p_0 - (1+\varepsilon/2) f(k-1).
        \\
        &
        \ge
        - (1+\varepsilon) f(k-1),
    \end{align*}
    where the last step uses \(f(k)/k \uparrow \infty\).

    The assumption \(k-1 \le f^{-1}( (1-\delta) \log n)\) implies that
    \(f(k-1) \le (1-\delta) \log n\) and \(k = O(\log n)\). Thus
    \[
        \log \pleastk
        \ge -(1+\varepsilon) (1-\delta) \log n
        \ge -(1-\delta/2) \log n.
    \]
    Thus \(n \pleastk \ge n^{\delta/2}\).
    We have \( n \pleastk/k \to \infty\).
    It follows from Lemma \ref{faimly:lem:all:appear:condition} that \(\gwn\) contains all possible trees of
    size at most \(k\) as fringe subtree whp.

    (ii):
    Let \(\tstar_{k-1}\) be the tree in which one node has degree \(k-1\) and all other nodes are
    leaves. Computations similar to above show that if \(k-1 \ge f^{-1}( (1+\delta) \log n)\),
    then \(n \pi(\tstar_{k-1}) \to 0\). Therefore \(\gwn\) does not contain \(\tstar_{k-1}\) whp.
\end{proof}

\begin{example*}[The discrete Gaussian distribution]
    \label{family:ge:i:super:exp}
    When \(p_i = c e^{-c' i^2}\) for some appropriate positive normalization constants \(c\) and \(c'\),
    we have \(L=0\), and Theorem \ref{family:thm:super:exp} applies.
    Then
    \[\frac {K_n}{\sqrt{\log(n)}} \inprob \frac{1}{\sqrt{c'}},\]
    as \(n \to \infty\).
\end{example*}

\begin{example*}[The Cayley trees]
A better example is the Galton-Watson tree with offspring distribution \(\xi \eql
\Po(1)\), i.e., the Cayley tree.
It has \(p_i = e^{-1}/i!\) and \(\log(1/p_i) \sim i \log(i)\).
It is easy to see that 
\[
    \frac{\pSetMaxSize \log \log n}{\log n} \inprob 1.
\]
Using \eqref{eq:fringe:pleast:lower} it is not difficult to verify that the
tail drops so fast that the least possible tree of size \(k\) is \(\tstar_{k-1}\).
This is a special case of the following general observation.
\end{example*}

\begin{lemma}
    \label{family:lem:Cayley:is:super:exp}
    Assume Condition \ref{cond:offspring}.
    If \(p_i > 0\) for all \(i \ge 0\) and \(
    p_{i}^{1/i} \downarrow 0\), then for \(k\) large enough,
    \(
        \pleastk
        =
        \p{\gw = \tstar_{k-1}}
        =
        p_{0}^{k-1} p_{k-1}
        .
    \)
    In particular, this is true for \(\xi \eql \Po(1)\). In the latter case we have
    \[
        \log \pleastk = \log(p_0^{k-1}p_{k-1}) = - k \log k (1 + O(1/k)).
    \]
\end{lemma}

It follows from Theorem \ref{family:thm:fringe:all:appear} that for \(\xi \eql \Po(1)\),
\[
    K_n \frac{\log \log n}{\log n} \inprob 1,
\]
as \(n \to \infty\).
This matches the result given by Theorem \ref{family:thm:super:exp}.
However, we can be more precise by applying the following theorem with \(\gamma = 1\).

\begin{theorem}
    \label{family:thm:better:Cayley}
    Assume Condition \ref{cond:offspring}.
    Also assume that
    \[
        \log (1/\pleast_{i}) = \gamma (\log i )(i + O(1))
        ,
    \]
    where \(\gamma > 0\) is a constant.
    Define
    \[m = \log n, \qquad m_1 = \log(m/\gamma), \qquad m_2 = \log m_1.\]
    Let \(k_n \to \infty\) be a sequence of positive integers.
    Then for all constants \(\delta > 0\), we have:
    \begin{enumerate}[label=(\roman*)]
        \item If
            \(
            k_n \le
            \frac{m}{\gamma (m_1 - m_2)}
            \left( 1-(1+\delta)\frac{m_2}{m_1(m_1-m_2)} \right)
            ,
            \)
            then whp \(\gwn\) contains
            all trees in \(\tsetpluskn\) as fringe subtrees.
        \item If
            \(
            k_n \ge
            \frac{m}{\gamma(m_1 - m_2)}
            \left( 1-(1-\delta)\frac{m_2}{m_1(m_1-m_2)} \right)
            ,
            \)
            then whp \(\gwn\) does not contain
            all trees in \(\tsetpluskn\) as fringe subtrees.
    \end{enumerate}
    Thus, as \(n \to \infty\),
    \[
        K_n \frac{\log \log n}{\log n} \inprob \frac{1}{\gamma},
    \]
    and more precisely,
    \[
        \left[
            \frac{K_n \gamma \left[
                \log \frac{\log n}{\gamma} - \log \log \frac{\log n}{\gamma}
            \right]}{\log n} - 1
            \right]
        \times
        \frac{(\log \log n)^2}{\log \log \log n} \inprob -1.
    \]
\end{theorem}

\begin{proof}
    Let \(k = k_n\).  For (i), we can show that \(n \pleastk /k\to \infty\).
    It follows from Lemma \ref{faimly:lem:all:appear:condition} that whp \(\gwn\) contains all
    trees of size at most \(k\) as fringe subtrees.

    For (ii), similarly we can show that \( n \pleastk \to 0.\) It follows from Theorem
    \ref{intro:thm:fringe:count} that whp there is at least one tree in \(\tsetplus{k}\) that does not
    appear in \(\gwn\) as fringe subtrees.
\end{proof}

\section{Non-fringe subtrees}

\label{sec:nfringe}

\newcommand{\veca}{{\bm a}}
\newcommand{\vecb}{{\bm b}}
\newcommand{\vecd}{{\bm d}}

In this section we prove Theorem \ref{intro:thm:nonfringe:count}, the concentration of
non-fringe subtree counts in conditional Galton-Watson trees.

Given a tree \(T\), let \(v(T)\) be the number of its internal nodes and let \(\ell(T)\) be the
number of its leaves.
Recall that \(\nftcount :=\sum_{u \in \gwn} \iverson{T \rootat \gw_{n,u}}\),
and that \( \xitildenvec := (\xitildenlist)\) is a uniform random rotation of the preorder degree sequence of
\(\gwn\).

To simplify the notation, write \(v := v(T)\) and \(\ell:=\ell(T)\).
By Lemma \ref{pre:lem:tree:deg:seq},
\(T\) has a preorder degree sequence of the form
\begin{align}
    & (\veca_1, 0, \veca_2, 0, \ldots, \veca_{\ell}, 0)
    :=
    \nonumber
    \\
    &
    \qquad
    \qquad
    ( a_{1,1}, a_{1,2}, \ldots, a_{1,r({1})}, 0,
    a_{2,1}, a_{2,2}, \ldots, a_{2,r({2})}, 0,
    \ldots,
    a_{\ell,1}, a_{2,2}, \ldots, a_{\ell,r({\ell})}, 0)
    \label{eq:deg:seq:nfringe:raw:form}
\end{align}
for positive integers \(r(1), r(2),\ldots, r(\ell)\) and that
\begin{equation}
        \sum_{s=1}^{\ell} r({s}) = v,
        \qquad
        a_{s,t} > 0,
        \qquad
        \sum_{s=1}^{\ell} \veca_s :=
        \sum_{s=1}^{\ell} \sum_{t=1}^{r(s)} a_{s,t}
        = v + \ell - 1.
    \label{eq:deg:seq:nfringe:raw:form:cond}
\end{equation}

Therefore, if \(T \rootat T'\), then \(T'\) has a preorder degree sequence of the form
\begin{equation}
    (\veca_{1}, \vecb_{1}, \veca_{2}, \vecb_{2}, \ldots, \veca_{\ell}, \vecb_{\ell})
    \label{eq:deg:seq:nfringe:extended:form}
\end{equation}
where \(\vecb_{1}, \ldots, \vecb_{\ell}\) are preorder degree sequences of some plane trees.
Thus each non-fringe subtree of shape \(T\) in \(\gwn\) corresponds to a segment of
\( \xitildenvec\) of the form of \eqref{eq:deg:seq:nfringe:extended:form}.
If none of the segments overlap with each other, then we can permute them into
the form
\(
    (\veca_{1}, \ldots, \veca_{\ell}, \vecb_{1}, \ldots, \vecb_{\ell}).
\)
Since \( \xitildenvec\) is permutation invariant, the new sequence still has the distribution of
\( \xitildenvec\). In other words, \(\nftcount\) is almost distributed
like the number of the patterns \( (\veca_1, \veca_2, \ldots, \veca_\ell)\) in \( \xitildenvec\).

The problem with this argument is that non-fringe subtrees can overlap.
But as shown later in this section,
under the assumptions of Theorem \ref{intro:thm:nonfringe:count}, the effect of such overlaps is negligible.

\newcommand{\cycD}[1]{\widetilde{\cD}_{#1}}
\newcommand{\cycDn}{\cycD{n}}

We will use
\(\cD_{n}\) to denote the set of preorder degree sequences of trees with size \(n\).
Let \(\cycDn\) be the set of sequences that are cyclic rotations of sequences in \(\cD_n\).
Given \(\vecd := (d_1,\ldots, d_n) \in \cycDn\),
let \( \deg_i(\vecd) := (d_i, d_{i+1}, \ldots, d_{i+k-1})\) such that
\(\deg_{i}(\vecd) \in \cD_{k}\) for some \(k \ge 1\), where the indices are all modulo \(n\).
Lemma \ref{pre:lem:tree:deg:seq} guarantees that such \(\deg_i(\vecd)\) exists and is
unambiguous.
Let \(T_i(\vecd)\) be the tree with the preorder degree sequence \(\deg_i(\vecd)\).

\subsection{Factorial moments}

Let \( (x)_r := x(x-1)\cdots (x-r+1)\). For a random variable \(X\), \(\e (X)_{r}\) is called the
\emph{\(r\)-th factorial moment} of \(X\). We give exact formulas for the
first and second factorial moments of \(\nftcount\) in this subsection.

\begin{lemma}
    \label{nonfringe:lem:expectation}
    Assume that \(\p{|\gw|=n}>0\).
    Let \(T\) be a tree.  We have
    \[
        \frac
        {
            \E{\nftcount}
        }
        {
            n
        }
        =
        \nfpi(T)
        \frac{
            \p{S_{n-v(T)} = n-v(T)-\ell(T)}
        }{\p{S_{n} = n-1}}.
    \]
\end{lemma}

\begin{proof}
    Let \(v := v(T)\) and \(\ell:=\ell(T)\).
    Let \(I_i = \iverson{T \rootat T_i(\xitildenvec)}\). Then \(\nftcount = \sum_{i=1}^{n} I_i\).
    By the permutation invariance of \( \xitildenvec\), we have \(\E{\nftcount}=
    \E{\sum_{i=1}^n I_i}= n \p{I_1=1}\).

    Recall that \(T\) has a preorder degree sequence of the form \( (\veca_1, 0, \ldots, \veca_{\ell}, 0)\)
    satisfying \eqref{eq:deg:seq:nfringe:raw:form:cond}.
    Let \(\cA \subseteq \cycDn\) be the set of sequences such that \( \xitildenvec \in
    \cA\) if and only if \(I_1 = 1\).
    In other words, \( \vecd := (d_1, d_2, \ldots, d_n) \in \cA\) if and only if \( \deg_1(\vecd) =
    (\veca_1, \vecb_1, \ldots, \veca_\ell, \vecb_{\ell})\) for some \(\vecb_1, \ldots,
    \vecb_{\ell}\) which are preorder degree sequences of trees.
    By permuting \( \deg_1(\vecd) \) into \( (\veca_1,\veca_2, \ldots,
    \veca_{\ell}, \vecb_{1}, \vecb_{2}, \ldots, \vecb_{\ell})\), we get a new sequence \(\vecd':= (d_1',
    d_2', \ldots, d_{n}') \in \cA'\)
        where
    \[
        \cA'
        :=
        \left\{
            (e_1, e_2, \ldots, e_n) \in \cycDn
            :
            (e_1, e_2, \ldots, e_v)
            = (\veca_{1}, \veca_{2}, \ldots, \veca_{\ell})
        \right\}
        .
    \]
    Such a permutation defines a mapping \(f:\cA \to \cA'\).

    For every \(\vecd' \in \cA'\), condition \eqref{eq:deg:seq:nfringe:raw:form:cond}
    implies that
    in \(\vecd'\) after \( (\veca_1,\ldots,\veca_\ell)\),
    there are at least \(\ell\) consecutive segments that are preorder degree sequences of trees, i.e.,
     there is a unique \(\vecd \in \cA\) with \(f(\vecd) = \vecd'\).
     Thus \(f\) is a one-to-one mapping. If
    \( \vecd' = f( \vecd)\), then
    \(
        \p{\xitildenvec = \vecd}
        =
        \p{\xitildenvec = \vecd'},
    \)
    since \( \xitildenvec\) is permutation invariant.
    Therefore we have
    \begin{align*}
        \p{I_1 = 1}
        &
        =
        \p{\xitildenvec \in \cA}
        =
        \p{\xitildenvec \in \cA'}
        .
    \end{align*}

    Recall that by Lemma \ref{pre:lem:tree:deg:seq:rotation},
    \( \xitildenvec \sim (\xilist|S_n=n-1)\), where
    \(\xilist\) are i.i.d.\ copies of \(\xi\) and \(S_n = \sum_{s=1}^{n} \xi_s\).
    We have
    \begin{align*}
        \p{\xitildenvec \in \cA'}
        &
        =
        \p{
            (\xitilden_1,\xitilden_2, \ldots, \xitilden_v) = (\veca_1, \veca_2, \ldots, \veca_{\ell})
        }
        \\
        &
        =
        \frac{
        \p{
            (\xi_1,\xi_2, \ldots, \xi_v) = (\veca_1, \veca_2, \ldots, \veca_{\ell})
            ,
            S_{n} = n-1
            }
        }
        {
            \p{S_n = n-1}
        }
        \\
        &
        =
        \p{T \rootat \gw}
        \frac{
            \p{S_{n-v}=n-v-\ell}
        }
        {
            \p{S_n = n-1}
        }
        ,
    \end{align*}
    where in the last step we use \(\sum_{s=1}^{\ell} \veca_s = v + \ell -1\).
\end{proof}

To compute \(\e (\nftcount)_2\), we enumerate all the cases that \(T\) can appear as overlapping
non-fringe subtrees by constructing a set of trees \(\Toplus\) as follows.
For trees \(T\), \(S\) and node \(v \in T\), let \(T'= \splay(T, v, S)\) denote tree \(T\) with subtree
\(T_v\) replaced by \(S\). Thus \(T_v' = S\) .
Let \(\cV(T)\) denote the internal nodes of \(T\).
Then define the collection
\[
    \Toplus = \left[ \bigcup_{v \in \cV(T):T_v \rootat T} \left\{ \splay(T, v, T) \right\} \right]
    \setminus \left\{T\right\}.
\]
Note that \(|\Toplus| < v(T)\).
Also note that given \(T' \in \Toplus\) we can always find a unique node \(v \in \cV(T)\) such that
\(T'=\splay(T,v,T)\).
See Figure \ref{nonfringe:fig:tree:glue} for an example of \(\Toplus\).

\begin{figure}[ht!]
  \centering
    \begin{tikzpicture}
    \node[anchor=south west,inner sep=0] at (0,0) {\includegraphics{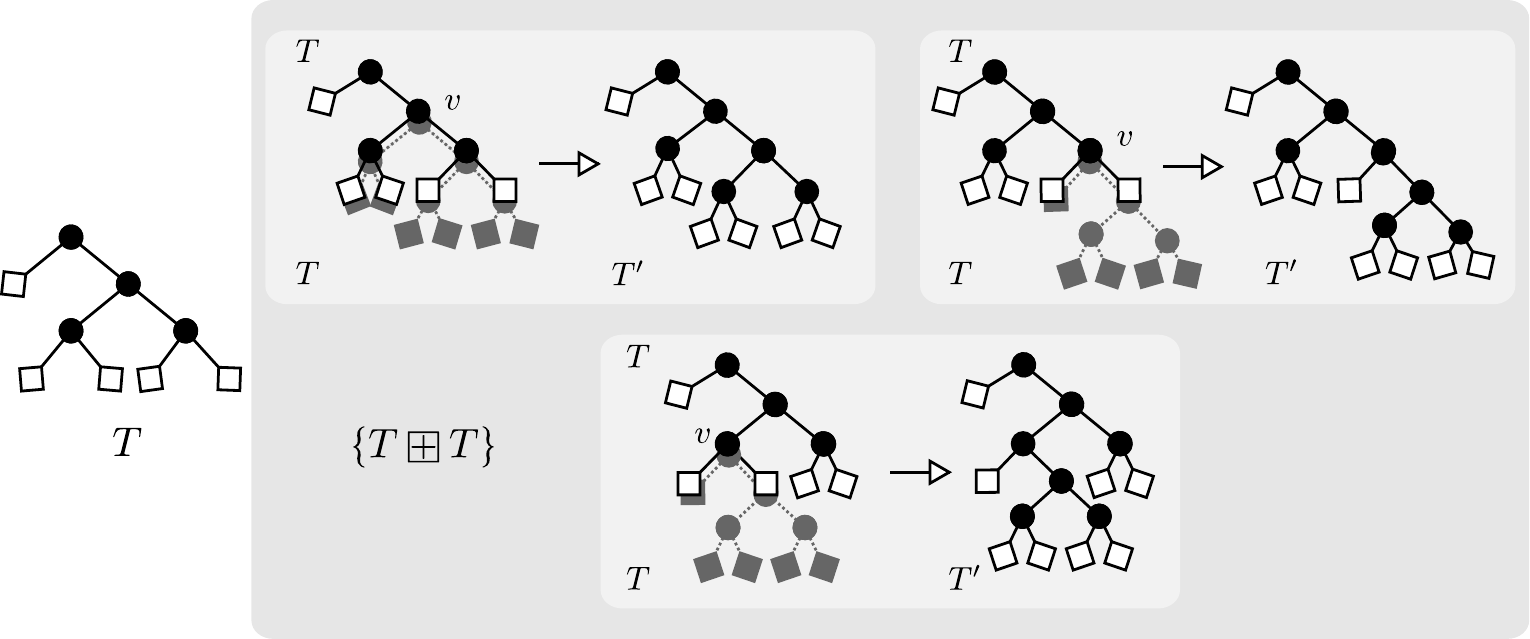}};
    \end{tikzpicture}
    \caption{An example of \(\Toplus\).}
    \label{nonfringe:fig:tree:glue}
\end{figure}

\begin{lemma}
    \label{nonfringe:lem:factorial:two}
    Assume that \(\p{|\gw| = n} > 0\).
    Let \(T\) be a tree.         We have
    \begin{align*}
        \E{ (\nftcount)_2 }
        &
        =
        n(n-2 v(T)+1)
        \nfpi(T)^2
        \frac{
            \p{S_{n-2v(T)} = n + 1 - 2(v(T)+\ell(T))}
        }{
            \p{S_{n}=n-1}
        }
        \\
        &
        +
        2n
        \sum_{T' \in \Toplus}
        \nfpi(T')
        \frac{
            \p{S_{n-v(T')} = n-v(T')-\ell(T')}
        }{\p{S_{n} = n-1}}.
    \end{align*}
\end{lemma}

\begin{proof}
    Let \(v = v(T)\) and \(\ell = \ell(T)\).
        Let \(I_i\) be defined as in the proof of Lemma \ref{nonfringe:lem:expectation}.
    Since \( I_1,\ldots, I_n \) are indicator random variables and permutation invariant, we have
    \[
        \E{(\nftcount)_2}
                                =
        \sum_{1 \le i\ne j \le n} \E{I_i I_j}
        = n \sum_{i = 2}^{n} \E{I_1 I_{i}}.
    \]

    The event \(I_{1} I_i =1\) happens if and only if \(T \rootat T_1(\xitildenvec)\) and \(T
    \rootat T_i(\xitildenvec)\) both happen.
    Thus instead of summing \(\E{I_1 I_{i}}\) over \(i\), we can sum
    \(\p{\xitildenvec = \vecd}\) over pairs \( (i, \vecd) \in \{2,\ldots,n\} \times \cycDn\) that satisfy
        \(T \rootat T_1(\vecd)\) and \(T \rootat T_i(\vecd)\), i.e.,
    \[\deg_1(\vecd) = (\veca_1, \vecb_1, \veca_2, \vecb_2, \ldots, \veca_\ell, \vecb_\ell),
        \qquad
        \text{and}
        \qquad
        \deg_i(\vecd) = (\veca_1, \vecb_1', \veca_2, \vecb_2', \ldots, \veca_\ell, \vecb_\ell'),\]
        where \( (\veca_1, 0, \ldots, \veca_{\ell}, 0) \) is the preorder degree sequence of \(T\) and
    \(\vecb_1, \vecb_1', \ldots, \vecb_\ell, \vecb_\ell'\) are
    preorder degree sequences of trees.
    Let \(\cA\) be the set of such pairs.
    Then \(\sum_{i\ge 2}\E{I_1I_i} = \sum_{(i,\vecd)\in\cA} \p{\xitildenvec = \vecd}\).

    For \(1 \le r \le n\), let \(\idxVecd{r}\) be the set of positions in \(\vecd\) that are occupied by \(\deg_r(\vecd)\), i.e.,
    \[
        \idxVecd{r}
        :=
        \left\{
            j \bmod n: r \le j < r + |\deg_r(\vecd)|
        \right\}
        .
    \]
    Let \(\idxVecdIn{1} \subseteq \idxVecd{1}\) be the set of positions in \(\vecd\) that are occupied
    by the parts of \(\deg_1(\vecd)\) that correspond to \(a_1, \ldots, a_\ell\). Let \(\idxVecdOut{1} = \idxVecd{1} \setminus \idxVecdIn{1}\). Define
    \(\idxVecdIn{i}\) and \(\idxVecdOut{i}\) accordingly.
    Let \(\cA' \subseteq \cA\) be the set of \( (i, \vecd)\) in \(\cA\)
    such that
    \[
        \cA'
        =
        \left\{
        (i,\vecd) \in \cA : \idxVecdIn{1} \cap \idxVecdIn{i} = \emptyset
        \right\}
        .
    \]
    Let \(\cA'' := \cA \setminus \cA'\).

    If \( (i, \vecd) \in \cA'' \), then \(\idxVecdIn{1} \cap \idxVecdIn{i} \ne \emptyset\).
    In other words, either \(T_i\) is fringe subtree of \(T_1\) and \(T_i\) is rooted at a node that corresponds
    to an internal node of \(T\) (regarding that \(T_1\) is a non-fringe subtree of the shape
    \(T\)), or vice versa.
    Thus there exists a \(T' \in \Toplus\) such that either \(T' \rootat T_1(\vecd)\) or \(T'
    \rootat T_i(\vecd)\).
    By symmetry, we have
    \begin{align*}
        \sum_{(i,\vecd) \in \cA''} \p{\xitildenvec=\vecd}
        &
        =
        2 \sum_{T' \in \Toplus} \p{T' \rootat T_1(\xitildenvec)}
        \\
        &
        =
        2 \sum_{T' \in \Toplus} \nfpi(T')
        \frac{\p{S_{n-v(T')}=n-v(T')-\ell(T')}}{\p{S_{n}=n-1}}
        ,
        \numberthis \label{eq:S:overlap}
    \end{align*}
    where the last step follows from Lemma \ref{nonfringe:lem:expectation}.

    Now consider \( (i, (d_1,\ldots,d_n)) \in \cA'\).
    Arrange \((d_1,\ldots,d_n)\) in a cycle.
    Paint the segment \(\deg_1((d_1,\ldots,d_n))\) red and the segment \(\deg_i((d_1,\ldots,d_n))\) blue.
    One of the three cases must be true:
        (i) \(\idxVecd{1} \cap \idxVecd{i} = \emptyset\) --- The red segment and the blue segment
            do not overlap.
        (ii) \(\idxVecd{i} \subseteq \idxVecd{1}\) --- The red segment contains the blue segment.
        (iii) \(\idxVecd{1} \subseteq \idxVecd{i}\) --- The blue segment contains the red segment.
    (Since \(\deg_1(\vecd)\) and \(\deg_i(\vecd)\) are both preorder degree sequences,
    if \(\idxVecd{1} \cap \idxVecd{i} \ne \emptyset\) then either (ii) or (iii) must
    happen. And since \(i \ne 1\) we cannot have \(\idxVecd{i} = \idxVecd{1}\).)
    Figure \ref{nonfringe:fig:tree:not:glue} gives examples of the three cases.
    \begin{figure}[ht!]
    \centering
        \begin{tikzpicture}
        \node[anchor=south west,inner sep=0] at (0,0) {\includegraphics{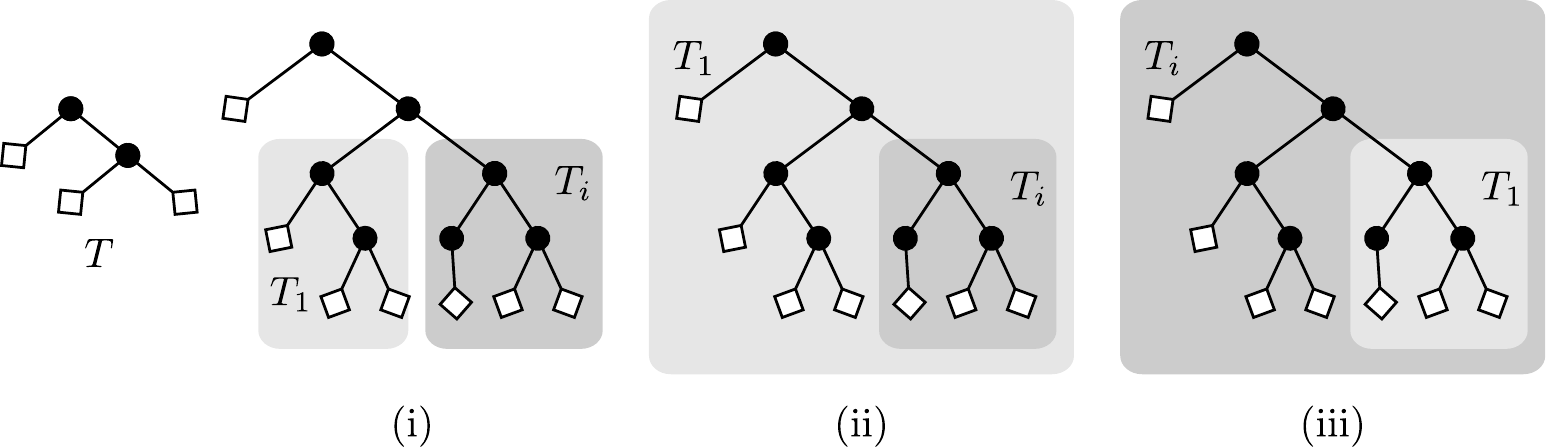}};
        \end{tikzpicture}
        \caption{Examples of three cases in \(\cA''\).}
        \label{nonfringe:fig:tree:not:glue}
    \end{figure}

    We permute \( (d_1,\ldots,d_n) \) as follows.
    For (i) and (ii), we first permute the red segment from
    \( (\veca_1,\vecb_1,\ldots,\veca_\ell,\vecb_\ell) \)
    to
    \( (\veca_1,\ldots,\veca_\ell,\vecb_1,\ldots,\vecb_\ell) \).
    Then we permute the blue segment of
    from
    \( (\veca_1,\vecb_1',\ldots,\veca_\ell,\vecb_\ell') \)
    to
    \( (\veca_1,\ldots,\veca_\ell,\vecb_1',\ldots,\vecb_\ell') \).
    It is clear this can be done in case (i).
    And it is not difficult to see that in case (ii) the positions that are occupied by the
    blue segment is completely contained by the positions that are occupied by \(b_{\ell'}\) for some
    \(1 \le \ell' \le \ell\).
    This means that \(T_i\) is a fringe subtree of \(T_1\) and the root of \(T_i\) does not
    correspond to an internal node of \(T\) (because \(T_1\) is a non-fringe subtree in the
    shape of \(T\)).
    So the first step of the permutation moves the blue segment but does not change its
    contents and we can carry out the second step without problem.

    In case (iii), we reverse the order of the two steps.  After this the starting position of the red segment may
    have changed. We rotate the new sequence such that the red segment still starts from position
    \(1\).

    In the end, we get a new pair \( (i', (d_1',\ldots,d_n')) \) such that
    \( (d_1',d_2',\ldots,d_v') = (\veca_1,\ldots,\veca_\ell)\)
    and
    \( (d'_{i'}, d'_{i'+1},\ldots,d'_{i'+v-1}) = (\veca_1,\ldots,\veca_\ell)\).
    Let \(\cB \subseteq \{v+1,\ldots,n\} \times \cycDn\) be the set of such pairs.
    The above permutation defines a mapping \(f:\cA'\to\cB\).
    Since given \( (i',\vecd') \in \cB\), we can without ambiguity recover the red segment and blue
    segment of \(\vecd'\), the mapping is reversible, i.e., \(f\) is one-to-one.
    Since \( \xitildenvec\) is permutation invariant, if \( (i',\vecd') = f(i, \vecd)\), then
    \(\p{ \xitildenvec = \vecd} = \p{\xitildenvec = \vecd'}\).
    Therefore
    \[
        \sum_{(i, \vecd) \in \cA'}
        \p{ \xitildenvec = \vecd}
        =
        \sum_{(i, \vecd) \in \cB}
        \p{\xitildenvec = \vecd}
        .
    \]

                Given \( (i, (d_1,\ldots, d_n)) \in \cB\), we can move the segment \( (d_{i}, \ldots,
    d_{i+v-1})\) to the position \(v+1\) to get a new sequence \( (d_1', \ldots, d_{n}') \in \cC\),
    where
    \[
        \cC:=\left\{
            (e_1,\ldots,e_n)\in \cycDn:
            (e_1, \ldots,e_{2v}) = (
            \veca_1,\ldots,\veca_\ell,
            \veca_1,\ldots,\veca_\ell
            )
        \right\}
        .
    \]
    Since there are \(n-2v+1\) possible values of \(i\), this permutation gives us a \(
    (n-2v+1)\)-to-one mapping \(h:\cB\to\cC\) and
    if \( \vecd' = h(i, \vecd)\), then
    \(\p{\xitildenvec = \vecd} = \p{\xitildenvec = \vecd'}\).

    We obtain as usual
    \begin{align*}
        \sum_{\vecd \in \cC}
        \p{\xitildenvec = \vecd}
        &
        =
        \p{
            (\xitilden_1,\ldots,\xitilden_{2v} )
            =
            (\veca_1,\ldots,\veca_{\ell},\veca_1,\ldots, \veca_{\ell})
        }
        \\
        &
        =
        \p{
            (\xi_1,\ldots,\xi_{2v} )
            =
            (\veca_1,\ldots,\veca_{\ell},\veca_1,\ldots, \veca_{\ell})
            ~|~
            S_n = n-1
        }
        \\
        &
        =
        \p{
            (\xi_1,\ldots,\xi_{2v} )
            =
            (\veca_1,\ldots,\veca_{\ell},\veca_1,\ldots, \veca_{\ell})
        }
        \\
        &
        \times
        \frac{
        \p{
            S_{n-2v} = (n-1)-2(v+\ell-1)
        }}
        {
            \p{S_{n} = n-1}
        }
        \\
        &
        =
        \nfpi(T)^2
        \frac{
        \p{
            S_{n-2v} = n+1-2(v+\ell)
        }}
        {
            \p{S_{n} = n-1}
        }         .
    \end{align*}

            It follows that         \begin{align*}
        \sum_{(i, \vecd) \in \cA'}
        \p{ \xitildenvec = \vecd}
        &
        =
        \sum_{(i', \vecd) \in \cB}
        \p{\xitildenvec = \vecd}
        =
        (n-2v+1) \sum_{\vecd \in \cC}
        \p{\xitildenvec = \vecd}
        \\
        &
        =
        (n-2v+1)
        \nfpi(T)^2
        \frac{
        \p{
            S_{n-2v} = n+1-2(v+\ell)
        }}
        {
            \p{S_{n} = n-1}
        } \numberthis \label{eq:S:separate}
        .
    \end{align*}

    The lemma follows by combining \eqref{eq:S:overlap} and \eqref{eq:S:separate} with the
    following:
    \begin{align*}
        \e{(\nftcount)_2}
        &
        =
        n \sum_{i=2}^n \E{I_{1} I_{i}}
        =
        n \sum_{(i,\vecd)\in\cA} \p{\xitildenvec = \vecd}
        \\
        &
        =
        n \sum_{(i,\vecd)\in\cA'} \p{\xitildenvec = \vecd}
        +
        n \sum_{(i,\vecd)\in\cA''} \p{\xitildenvec = \vecd}
        .
        \qedhere
    \end{align*}
\end{proof}

\subsection{Sequence of non-fringe subtrees}

\label{sec:nfringe:asympt}

Let \(T_{n}\) be a sequence of trees.
Let \(v_n  := v(T_n)\) and \(\ell_n  := \ell(T_n)\).
In this subsection we prove Theorem \ref{intro:thm:nonfringe:count}, the concentration of
\(\nftncount\).
First we narrow down the scopes of \(v_n\) and \(\ell_n\).
\begin{lemma}
    \label{nonfringe:lem:v:upper}
    Assume Condition \ref{cond:offspring}.
    There exists a constant \(C_1 > 0\) such that
    \[
        \sup_{T:v(T) > C_1 \log n} n \nfpi(T) \to 0.
    \]
    Thus if \(v_n > C_1 \log n\), then \(\nftncount \inprob 0\) as \(n \to \infty\).
\end{lemma}

\[
    \sup_{T:v(T) > C_1 \log n}
    \nftcount \inprob 0.
\]
The whole \(\gwn\) is of size \(n\) and is a non-fringe subtree itself.

\begin{proof}
    Recall that \(\pmax := \max_{i \ge 0}p_i\).
    We have \(\p{T \rootat \gw} \le \pmax^{v(T)}\).
    Since \(\pmax < 1\), if \(C_1 = 2 /\log(1/\pmax)\), then
    \[
        \sup_{T:v(T) > C_1 \log n}
        n \nfpi(T)
        \le
        n \pmax^{C_1 \log n}
        = n n^{-2} \to 0.
    \]

    By Lemma \ref{nonfringe:lem:expectation}, we have
    \begin{align*}
        \E{\nftncount}
                =
        n \p{T_n \rootat \gw}
        \frac{\p{S_{n-v_n}=n-v_n-\ell_n}}{\p{S_{n}=n-1}}
                        \le
        n \pmax^{v_n}\p{S_{n}={n-1}}^{-1}.
    \end{align*}
    It follows from the well-known local limit theorem (see, e.g., \citet{Kolchin1986random}[thm.\ 1.4.2])
    that \(\p{S_{n}=n-1}= \Theta(n^{-1/2})\). Therefore
    if \(v_n > C_1 \log n\), the above expression is at most
    \[
        n \pmax^{v_n} \bigO{n^{1/2}}
        \le
        n
        n^{-2}
        O(n^{1/2})
        \to 0.
        \qedhere
    \]
\end{proof}

\begin{lemma}
    \label{nonfringe:lem:l:upper}
    Assume Condition \ref{cond:offspring}
    and that \(v_n = O(\log n)\).
    There exists a sequence \(b_n = o(n^{1/2} \log n)\) such that if \(\ell_n > b_n\), then
    \(n \nfpi(T_n) \to 0\) and \(\nftncount \inprob 0\) as \(n \to \infty\).
\end{lemma}

\begin{proof}
    \citet[thm.\ 19.7]{janson2012simply} showed that
    if \(\V{\xi} < \infty\) then there exists a sequence \(b_n' = o(n^{1/2})\) such that
    the maximum degree of \(\gwn\) is less than \(b_n'\) whp.
    Let \(b_n = v_n b_n' = o(n^{1/2} \log n)\).
    If \(\ell_n \ge b_n\), then \(T_n\) must contain at least one node with degree at least
    \(b_n'\). The probability of this event goes to zero.
\end{proof}

\begin{lemma}
    \label{nonfringe:lem:expectation:asympt}
    Assume Condition \ref{cond:offspring}.
    If \(|T_n| = o(n)\), then \(\E{\nftncount} /(n \nfpi(T_n)) \to 1\) as \(n \to \infty\).
\end{lemma}

\begin{proof}
    \(|T_n|=o(n)\) implies that \(v_n = o(n)\) and \(\ell_n = o(n)\).
    Therefore it follows Lemma \ref{pre:lem:prob:sum:of:iid} and \ref{nonfringe:lem:expectation} that
    \[
        \frac{\E{\nftncount}}{n \nfpi(T_n)}
        =
        \frac{\p{S_{n-v_n}=n-v_n-\ell_n}}{\p{S_n = n-1}}
        \to 1.
        \qedhere
    \]
\end{proof}

\begin{lemma}
    \label{nonfringe:lem:factorial:two:asympt}
    Assume Condition \ref{cond:offspring}.
    If \(n \nfpi(T_n) \to \infty\), then \(\e (\nftncount)_2 / (n \nfpi(T_n))^2 \to 1\) as \(n \to
    \infty\).
\end{lemma}

\begin{proof}
    Let \(v=v_n\), \(\ell=\ell_n\) and \(T = T_n\).
    Let \(C_1\) and \(b_n\) be defined as in Lemma \ref{nonfringe:lem:v:upper} and
    \ref{nonfringe:lem:l:upper} respectively.
    We can assume that \(v \le C_1 \log n = o(n)\) and \(\ell \le b_n = o(n)\). Otherwise we cannot have
    \(n \nfpi(T) \to \infty\) by these two lemmas.
    If \(T' \in \Toplus\), then \(v(T') < 2 v = o(n)\) and \(\ell(T') < 2 \ell = o(n)\).
    Therefore, it follows from Lemma \ref{pre:lem:prob:sum:of:iid} and \ref{nonfringe:lem:factorial:two} that
    \begin{align*}
        \e (\nftcount)_2
        &
        =
        n(n-2v+1) \nfpi(T)^2
        \frac{
            \p{S_{n-2v}=n+1-2(v+\ell)}
        }{
            \p{S_{n}=n-1}
        }
        \\
        &
        \qquad
        +
        2 n
        \sum_{T' \in \Toplus}
        \nfpi(T')
        \frac{
            \p{S_{n-v(T')}= n- v(T')-\ell(T')}
        }{
            \p{S_{n}=n-1}
        }
        \\
        &
        =
        (1+o(1)) (n\nfpi(T))^2
        +
        O(n)
        \sum_{T' \in \Toplus}
        \nfpi(T')
        .
    \end{align*}
    Thus it suffices to show that
    \(
    n \sum_{T' \in \Toplus} \nfpi(T')
        = o(n \nfpi(T))^{2}.
    \)

    Consider the superset \(\cA\) of \(\Toplus\) that contains trees which
    can be obtained by replacing a proper non-leaf subtree of \(T\) with another copy of \(T\). (We do
    not restrict where this replacement can happen as in the definition of \(\Toplus\).)
    Note that \(|\cA| = v-1\), since \(T\) has \(v\) internal nodes and one of them is the
    root.

    If \(T' \in \cA\), then \(T'\) contains \(T\) as a fringe subtree.
    Thus \(\nfpi(T') \le \nfpi(T)\).
    In the case that \(v\) is bounded, we have
    \[
        n \sum_{T' \in \cA} \p{T' \rootat \gw}
        \le n v \nfpi(T) = O(n \nfpi(T)) = o(n \nfpi(T))^2.
    \]
    Thus we can assume that \(v \to \infty\).

    For \(T' \in \cA\), if \(T'\) has at least \(3v/2\) internal nodes, call \(T'\) \emph{big}, otherwise call it
    \emph{small}. Let \(\ToplusB\) and \(\ToplusS\) be the sets of big and small trees in
    \(\cA\) respectively.

    If \(T' \in \ToplusB\), then besides internal nodes that correspond to internal nodes of \(T\),
    \(T'\) contains at least \(v/2\) extra internal nodes. So we have
    \(
        \p{T' \rootat \gw} \le \nfpi(T) \pmax^{v/2}.
    \)
    Since \(v \to \infty\) and \(\pmax<1\), \(v \pmax^{v/2} = o(1)\).
    Using that \(|\cA| < v\), we have
    \[
        n \sum_{T' \in \ToplusB} \p{T' \rootat \gw}
        \le n v \nfpi(T) \pmax^{v/2}
        = o(n \nfpi(T)).
    \]

    Let \(T_{i,j}\) be a fringe subtree in \(T\) whose root is at depth \(i\) and is the \(j\)-th
    node of this level. If replacing \(T_{i,j}\) with a copy of \(T\) makes a new tree \(T'_{i}\)
    that has strictly less than \(3v/2\) internal nodes, then \(T_{i,j}\) must contain more than \(v/2\) internal nodes.
    Therefore, for each \(i\), there is at most one possible such \(j\).
    For an example of \(T_{i}'\), see Figure \ref{nonfringe:fig:tree:small}.

    \begin{figure}[ht!]
    \centering
        \begin{tikzpicture}
        \node[anchor=south west,inner sep=0] at (0,0) {\includegraphics{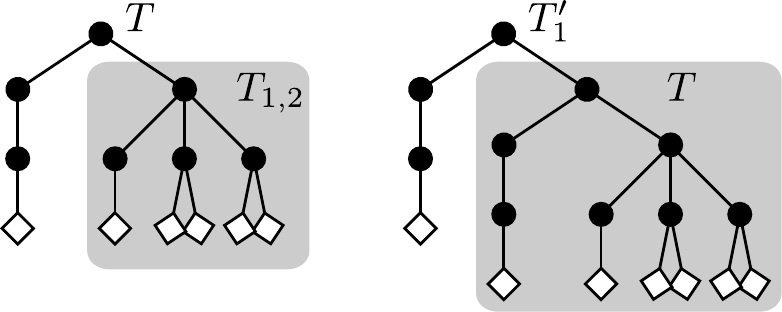}};
        \end{tikzpicture}
        \caption{An example of \(T_1'\) for \(T\) with \(7\) internal nodes.}
        \label{nonfringe:fig:tree:small}
    \end{figure}

    As \(T\) has \(v\) internal nodes, there are at most \(v-1\) possible \(i\) that can make
    \(T_{i,j}\) a proper and non-leaf subtree.
    Since \(T'_{i}\) has at least \(i\) internal nodes besides these in the copy of
    \(T\) that replaced \(T_{i,j}\),
    we have \(\nfpi(T'_{i}) \le \nfpi(T) \pmax^{i}\).
    In summary, we have
    \[
        n \sum_{T' \in \ToplusS} \nfpi(T')
        \le
        n \sum_{i=1}^{v} \nfpi(T'_{i})
        \le
        n \sum_{i=1}^{v} \nfpi(T) \pmax^{i}
        \le
        O(n \nfpi(T)).
    \]

    Therefore,
    \begin{align*}
        n \sum_{T' \in \Toplus} \p{T' \rootat \gw}
        &
        \le
        n \sum_{T' \in \ToplusB} \p{T' \rootat \gw}
        +
        n \sum_{T' \in \ToplusS} \p{T' \rootat \gw}
        \\
        &
        =
        o(n \nfpi(T))
        +
        O(n \nfpi(T))
                        = o(n \nfpi(T))^2
        .
        \qedhere
    \end{align*}
\end{proof}

The method used in proving Lemma \ref{nonfringe:lem:factorial:two} and
\ref{nonfringe:lem:factorial:two:asympt} can be adapted to show the
following lemma for higher factorial moments.
We leave the proof for interested readers.

\begin{lemma}
    \label{nonfringe:lem:factorial:high}
    Assume Condition \ref{cond:offspring}.
    If \(n \nfpi(T_n) \to \infty\), then for all fixed \(r \ge 2\),
    \[\frac {\e (\nftncount)_r} {\left(n \nfpi(T_n)\right)^{r}} \to 1,\] as \(n \to \infty\).
\end{lemma}

The condition that \(n \nfpi(T_n) \to \infty\) is necessary, as shown by the following lemma.
\begin{lemma}
    \label{nonfringe:lem:line}
    Assume Condition \ref{cond:offspring}.
    Let \(L_{h(n)}\) be a chain (complete \(1\)-ary tree) of height \(h(n)\).
    Let \(X_n := \nftreecount_{L_{h(n)}}(\gwn)\).
    If \(n \p{L_{h(n)} \rootat \gw} \to \mu \in (0, \infty)\) as \(n \to \infty\), then
    \[
        \e X_n \to \mu,
        \qquad
        \V{X_n} \to
        \mu \frac{1+p_1}{1-p_1},
        \qquad
        \E{(X_n - \e X_n)^{3}}
        \to
        \mu
        \frac{3p_1^{2}+2p_1+1}{(1-p_1)^2}
        .
    \]
    As a result, \(\liminf_{n \to \infty} \distTV{X_n, \Po(\mu)} > 0\).
\end{lemma}

\begin{proof}
    Let \(h = h(n)\).  Since \(n \p{L_h \rootat \gw} = n p_1^{h} \to \mu \in (0,\infty)\),
    we have \(h = \log_{1/p_1} n + O(1)\).
    \(L_h\) has \(h\) internal nodes and one leaf.
    Thus it follows from Lemma \ref{nonfringe:lem:expectation} and \ref{pre:lem:prob:sum:of:iid}
    that
    \begin{align*}
        \e X_n
        =
        n \nfpi(L_h)
        \frac{\p{S_{n-h}=n-h-1}}{\p{S_{n} = n-1}}
        \to \mu
        .
    \end{align*}

        Since
    \(
    \{L_h \oplus L_h \}
        = \{L_{h+i}: 1 \le i \le h-1
        \}
        ,
    \)
    by Lemma \ref{pre:lem:prob:sum:of:iid},
    \begin{align*}
        \zeta_1
        &
        :=
        2n \sum_{T' \in \{L_h \oplus L_h\}}
        \nfpi(T')
        \frac{
            \p{S_{n-v(T')} = n- v(T')-\ell(T')}
        } {
            \p{S_{n}=n-1}
        }
        \\
        &
        =
        2n \sum_{i=1}^{h-1}
        \nfpi(L_{h+i})
        \frac{
            \p{S_{n-h-i} = n-h-i-1}
        } {
            \p{S_{n}=n-1}
        }
        \\
        &
        =
        (1+o(1)) 2n
        \sum_{i=1}^{h-1} p_1^{h+i}
        =
        (1+o(1)) 2n p_1^h
        \sum_{i=1}^{h-1} p_1^{i}
                        \to 2 \mu \frac{p_1}{1-p_1}
        .
    \end{align*}
    We also have by Lemma \ref{pre:lem:prob:sum:of:iid},
    \begin{align*}
        \zeta_2
        &
        :=
        n(n-2h) \nfpi(L_h)^2
        \frac{\p{S_{n-2h}=n+1-2(h+1)}}{\p{S_n=n-1}}
        \to \mu^2.
    \end{align*}
    Therefore, it follows from Lemma \ref{nonfringe:lem:factorial:two} that
    \[
        \e (X_{n})_2 = \zeta_1 + \zeta_2
        \to
        2 \mu \frac{p_1}{1-p_1} + \mu^2.
    \]
    Thus
    \[
        \V{X_n} = \E{(X_{n})_2} + \E{X_n} - \E{X_{n}}^2 \to \mu \frac{1+p_1}{1-p_1}.
    \]
    So we have
    \[
        \frac{\V{X_n}}{\E{X_n}} \to \frac{1+p_1}{1-p_1} > 1.
    \]

    With an argument similar to Lemma \ref{nonfringe:lem:factorial:two}, we can compute
    \(\E{(X_n)_3}\), which yields
    \[
        \E{(X_n - \e X_n)^{3}}
        \to
        \mu
        \frac{3p_1^{2}+2p_1+1}{(1-p_1)^2}
        .
    \]
    Since 
    \begin{align*}
        \E{|X_n - \e X_n|^3} 
        &
        = 
        2 \E{|X_n - \e X_n|^3 \times \iverson{X_n < \e X_n}}
        +
        \E{(X_n - \e X_n)^3} 
        \\
        &
        \le
        2 (\e X_n)^3
        +
        \E{(X_n - \e X_n)^3} 
        ,
    \end{align*}
    the above limit implies that
    \[
        C 
        :=
        \limsup_{n \to \infty} 
        \E{|X_n - \e X_n|^3} 
        < \infty.
    \]

    Now we can finish by following the method of \citet[thm.\ 3B]{Barbour1992poisson}.
    Let \(Z_n\eql \Po(\e X_n)\) be a coupling of \(X_n\) that minimizes \(\p{Z_n \ne X_n}\).
    Therefore we have \(\distTV{X_n, \Po(\e X_n)} = \p{Z_n \ne X_n}\).
    Thus
    \begin{align*}
    \V{X_n} - \E{X_n} 
    & = \E{(X_n-\e X_n)^2} - \E{(Z_n-\e X_n)^2} \\
    & = \E{[(X_n-\e X_n)^2-(Z_n-\e X_n)^2] \times \iverson{X_n \ne Z_n}} \\
    & \le \E{(X_n-\e X_n)^2 \times \iverson{X_n\ne Z_n}} \\
    & \le P(X_n \ne Z_n)^{1/3} (\E{|X_n-\e X_n|^3})^{2/3},
    \end{align*}
    where in the last step we use Hölder's inequality \citep[pp.\ 129]{Gut2013}.
    So
    $$
    \distTV{X_n, \Po(\e X_n)} =
    \p{X_n\ne Z_n} \ge 
    \left(
    \frac{\V{X_n}-\e X_n}{E(|X_n-\e X_n|^3))^{2/3}}
    \right)^{3}
    .
    $$
    Therefore
    \[
        \liminf_{n \to \infty} 
        \distTV{X_n, \Po(\e X_n)}
        \ge \frac{1}{C^{2}} \left( \frac{1+p_1}{1-p_1} \right)^{3} 
        > 0
        .
    \]
    Since \(\e X_n \to \mu\), we also have
    \(
        \liminf_{n \to \infty} 
        \distTV{X_n, \Po(\mu)}
        > 0
        .
    \)
\end{proof}

\begin{proof}[Proof of Theorem \ref{intro:thm:nonfringe:count}]
    (i): By Lemma \ref{nonfringe:lem:v:upper} we can assume \(v_n \le C_1 \log n\) for \(n\) large.
    Otherwise we have \(\p{\nftncount \ge 1} \le \e \nftncount \to 0\) along the subsequence that \(v_n > C_1 \log n\).
    Similarly we can assume that \(\ell_n = o(n)\) by Lemma \ref{nonfringe:lem:l:upper}.

    Then by lemma \ref{nonfringe:lem:expectation:asympt}, \(\e \nftncount \sim n \nfpi(T_n)\).
    Thus \(n \nfpi(T_n) \to 0\) implies \(\nftncount \inprob 0\).

    (ii): We can assume that \(v_n \le C_1 \log n\) and \(\ell_n = o(n)\) by Lemma
    \ref{nonfringe:lem:v:upper} and \ref{nonfringe:lem:l:upper}. Otherwise we cannot have
    \(n \nfpi(T_n) \to \infty\).
    Therefore, by Lemma \ref{nonfringe:lem:expectation:asympt}, \(\e \nftncount \sim n \nfpi(T_n) \to
    \infty\).
    By Lemma \ref{nonfringe:lem:factorial:two:asympt}, we have
    \(
        \e (\nftncount)_2 = (1+o(1)) (n\nfpi(T_n))^2.
    \)
    Therefore,
    \begin{align*}
        \V{\nftncount}
        &
        =
        \E{(\nftncount)_2}
        + \E{\nftncount}
        - (\E{\nftncount})^2
        \\
        &
        =
        (1+o(1))(n \nfpi(T_n))^2
        +
        (1+o(1))(n \nfpi(T_n))
        -
        (1+o(1))(n \nfpi(T_n))^2
        \\
        &
        =
        o(n \nfpi(T_n))^2
        .
    \end{align*}
    Thus \(\nftncount/(n \nfpi(T_n)) \inprob 1\).
\end{proof}

\subsection{Complete $r$-ary non-fringe subtrees}

\label{sec:nfringe:app}

Theorem \ref{intro:thm:nonfringe:count} allows us to find the maximal complete \(r\)-ary non-fringe subtree
in \(\gwn\).
We omit the proofs of the following results due to their similarities to Lemma \ref{family:lem:r:ary} and
\ref{family:lem:fringe:chain}:
\begin{lemma}
    \label{nonfringe:lem:r:2}
    Assume Condition \ref{cond:offspring} and that \(p_r > 0\) for some \(r \ge 2\).
    Let \(\widenonf{H}_{n,r}\) be the height of the maximal complete \(r\)-ary
    non-fringe subtree in \(\gwn\).
    Then as \(n \to \infty\),
    \[
        \frac{\widenonf{H}_{n,r}}{\log_r(\log n)} \inprob 1.
        \]
\end{lemma}

\begin{lemma}
    \label{nonfringe:lem:r:1}
    Assume Condition \ref{cond:offspring} and that \(p_1 > 0\).
    Let \(\widenonf{H}_{n,1}\) be the height of the maximal chain (complete \(1\)-ary)
    non-fringe subtree in \(\gwn\).
    Then  as \(n \to \infty\),
    \[
        \frac{\widenonf{H}_{n,1}}{\log_{1/p_1} n} \inprob 1.
    \]
\end{lemma}

\begin{example*}[The binary tree]
    Recall that when \(p_0=p_2=1/4\) and \(p_1 = 1/2\), \(\gwn\) is equivalent to a uniform random
    binary tree of size \(n\). It follows from Lemma \ref{nonfringe:lem:r:1} that
    \(\widenonf{H}_{n,1}/\log_2 n \inprob
    1\). This result was previously proved by \citet{Devroye1999properties}.
\end{example*}

\section{Open questions}

\label{sec:open}

Theorem \ref{intro:thm:nonfringe:count} shows that if \(n \nfpi(T_n):=n \p{T_n \rootat \gw} \to \infty\),
then \(\nftncount/n \nfpi(T_n) \inprob 1\).
We believe that with an extra assumption that \(v(T_n) \to \infty\),
it is also true that \( (\nftncount-n \nfpi(T_n))/\sqrt{n \nfpi(T_n)}\) converges in
distribution to a normal distribution.
(\citet[thm.\ 1.9]{Janson2014Asym} has shown that this is indeed the case when \(v(n)\) is bounded.)

Theorem \ref{intro:thm:fringe:set:count} generalizes Theorem \ref{intro:thm:fringe:count} by considering
the number of fringe subtrees whose shapes belong to a set of trees \(\tsetkn\) instead of being a single tree
\(T_n\). It may be possible to generalize Theorem \ref{intro:thm:nonfringe:count} in similar way, i.e.,
we consider the non-fringe subtrees whose shapes belong to a set of trees \(\tsetkn\) instead of being a single tree
\(T_n\).

Another problem may be of interest is to get a non-fringe version of Theorem
\ref{family:thm:fringe:all:appear}, i.e., what are the sufficient conditions for all (or not all) trees of
size at most \(k\) to appear in \(\gwn\) as non-fringe subtrees.

Let \(T\) be a tree and \(v\) be a node of \(T\).  Recall that \(T_v\) denotes the fringe subtree rooted at
\(v\).  If by removing some or none the subtrees of \(T_{v}\), we can make it isomorphic to another
tree \(T'\), then we say that \(T\) contains an \emph{embedded subtree} of the shape \(T'\) at
\(v\). A more challenging open question is to determine the size of the maximum complete \(r\)-ary
embedded subtree in large conditional Galton-Watson trees.

\bibliographystyle{abbrplainnat}
\bibliography{citation}

\end{document}